\documentclass[12pt,a4paper,leqno]{amsart}
\usepackage[utf8]{inputenc}
\usepackage[T1]{fontenc}
\usepackage[UKenglish]{babel}
\usepackage[a4paper,margin=1in]{geometry}

\usepackage{amsmath,amsthm,amsfonts}
\usepackage{mathtools}
\usepackage{stix2,dsfont}

\usepackage{graphicx} 
\usepackage{url}

\usepackage[section]{placeins}

\usepackage{array}
\newcolumntype{?}{!{\vrule width 0.7pt}}
\newcolumntype{T}[1]{S[table-format=#1]}
\newcolumntype{U}[1]{S[table-format=#1,
                       round-mode=places, 
                       round-precision=2]}

\usepackage{hyperref} 
\hypersetup{hidelinks,colorlinks=true,linkcolor=black,citecolor=black}

\theoremstyle{plain}
\newtheorem{theorem}{Theorem}
\newtheorem{lemma}[theorem]{Lemma}

\newtheorem{corollary}[theorem]{Corollary}

\theoremstyle{definition}
\newtheorem{remark}[theorem]{Remark}
\newtheorem{definition}[theorem]{Definition}

\newtheorem*{assumption}{Assumption}

\numberwithin{theorem}{section}
\numberwithin{equation}{section}

\usepackage{enumitem}

\DeclareMathOperator*{\argmax}{arg\,max}

\renewcommand{\leq}{\leqslant}

\renewcommand{\geq}{\geqslant}

\newcommand{\BF}{\mathcal{BF}}
\newcommand{\CBF}{\mathcal{CBF}}
\newcommand{\SBF}{\mathcal{SBF}}
\newcommand{\CM}{\mathcal{CM}}

\newcommand{\Dcal}{\mathcal{D}}
\newcommand{\Lscr}{\mathscr{L}}
\newcommand{\Acal}{\mathcal{A}}
\newcommand{\Bcal}{\mathcal{B}}
\newcommand{\Fcal}{\mathcal{F}}
\newcommand{\Kcal}{\mathcal{K}}

\newcommand{\nat}{\mathds{N}}
\newcommand{\real}{\mathds{R}}

\newcommand{\Ee}{\mathds{E}}
\newcommand{\Pp}{\mathds{P}}
\newcommand{\I}{\mathds{1}}

\newcommand{\omu}[3]{\overset{\textrm{#1}}{\underset{\textrm{#3}}{#2}}}
\newcommand{\pomu}[3]{\overset{\textrm{\phantom{#1}}}{\underset{\textrm{\phantom{#3}}}{#2}}}

\newcommand{\rli}[3]{\mathbin{\sideset{_{\,#1}^{\mathrm{R}}}{_{#2}^{#3}}{\mathop{\mathrm{I}}}}}  

\newcommand{\cei}[3]{\mathbin{\sideset{_{~\,#1}^{\mathrm{Ce}}}{_{#2}^{#3}}{\mathop{\mathrm{I}}}}}
\newcommand{\rld}[3]{\mathbin{\sideset{_{\,#1}^{\mathrm{R}}}{_{#2}^{#3}}{\mathop{\mathrm{D}}}}}
\newcommand{\mad}[3]{\mathbin{\sideset{_{#1}^{\mathrm{M}}}{_{#2}^{#3}}{\mathop{\mathrm{D}}}}}
\newcommand{\cd}[3]{\mathbin{\sideset{_{\,#1}^{\mathrm{C}}}{_{#2}^{#3}}{\mathop{\mathrm{D}}}}} 
 
\newcommand{\ced}[3]{\mathbin{\sideset{_{~\,#1}^{\mathrm{Ce}}}{_{#2}^{#3}}{\mathop{\mathrm{D}}}}} 

\DeclareMathOperator{\Ext}{Ext}

\begin{document}

\title[Bernstein Fractional Derivatives and Processes]{Bernstein Fractional Derivatives: Censoring and Stochastic Processes}

\author[D.~Berger]{David Berger}
\address[D.~Berger]{
	TU Dresden, 
	Fakult\"at Mathematik, 
	Institut f\"ur Mathematische Stochastik, 
	01062 Dresden, Germany. 
	E-Mail: \textnormal{david.berger2@tu-dresden.de}}

\author[C.~Li]{Cailing Li}
\address[C.~Li]{
	TU Dresden, 
	Fakult\"at Mathematik, 
	Institut f\"ur Mathematische Stochastik, 
	01062 Dresden, Germany. 
	E-Mail: \textnormal{cailingli.math@gmail.com}}

\author[R.L.~Schilling]{Ren\'e L.\ Schilling}
\address[R.L.~Schilling]{
	TU Dresden, 
	Fakult\"at Mathematik, 
	Institut f\"ur Mathematische Stochastik, 
	01062 Dresden, Germany. 
	E-Mail: \textnormal{rene.schilling@tu-dresden.de}}



\thanks{This paper is based on Cailing Li's PhD thesis \cite{2023_Li}, which was written at Technische Universit\"at Dresden.} 
\thanks{Our research was supported by the \emph{TU Dresden Graduate Academy}, the \emph{6G-life} project (BMBF 16KISK001K) and the Dresden--Leipzig \emph{ScaDS.AI} centre.}
 
\maketitle

\begin{abstract}
	We define censored fractional Bernstein derivatives on the positive half-line based on the Bernstein--Riemann--Liouville fractional derivative. The censored fractional derivative turns out to be the generator of the censored decreasing subordinator $S^c = (S_t^c)_{t\geq 0}$, which is obtained either via a pathwise construction by removing those jumps from the decreasing subordinator $(x-S_t)_{t\geq 0}$, $x>0$, that drive the path into negative territory, or via the Hille--Yosida theorem. Then we show that the censored decreasing subordinator has only finite life-time, and we identify various probability distributions related to $S^c$.
\end{abstract}


\section{Introduction}\label{intro}
In this paper we focus on a special class of L\'evy processes, the so-called subordinators, see e.g.\ \cite{1996_Bertoin,1999_Bertoin,1999_Satoa}. A \textbf{L\'evy process} is a stochastic process with c\`adl\`ag (right--continuous, finite left limits) paths and independent and stationary increments. A \textbf{subordinator} is a L\'evy process $S=(S_t)_{t\geq 0}$ with $S_0=0$ and a.s.\ increasing paths. The process $S$ is uniquely defined by its Laplace transform, which is given by 
\begin{equation}\label{intro-e02}
    \Ee\left(e^{- \lambda S_t}\right) = e^{-t f(\lambda )}, \quad  \lambda>0.
\end{equation}
The characteristic exponent $f$ is a \textbf{Bernstein function}. A Bernstein function $f\in\BF$ can be expressed by
\begin{equation}\label{intro-e04}
    f(\lambda) = a + b\lambda + \int_0^\infty (1-e^{-\lambda x})\,\mu(dx), \quad  \lambda>0,
\end{equation}
where $a\geq 0$, $b\geq 0$ is the drift and $\mu$ is the L\'evy (or jump) measure, i.e.\ a Borel measure on $(0,\infty)$ such that $\int_0^\infty \min\{1,x\}\,\mu(dx)<\infty$.

Every subordinator is a Feller process on $C_\infty[0,\infty) = \left\{u\in C[0,\infty)\::\: \lim_{x\to\infty} u(x) = 0\right\}$, and its infinitesimal generator is given by 
\begin{align*}
	\Acal u(x)
	= au(x) + b\cdot \frac d{dx}u(x)+\int_{(0,\infty)} \left(u(x+t)-u(x)\right)\mu(dt) 
\end{align*}
for suitable functions $u:[0,\infty)  \to \real$.

We will restrict ourselves to so-called \textbf{complete Bernstein functions} $f\in\CBF$ which are of the form \eqref{intro-e04}, but the jump measure $\mu(dt)$ is absolutely continuous w.r.t.\ Lebesgue measure and the density $m(t) = \frac{\mu(dt)}{dt}$ is a completely monotone function. Among the most prominent examples of a complete Bernstein functions are the fractional powers $x^\alpha$, $\alpha\in (0,1)$. Our standard reference on (complete) Bernstein functions is the monograph \cite{2012_Schilling}, see also the appendix. It is obvious from the results in the appendix that most of the results presented for the class $\CBF$ will remain valid for special Bernstein functions.

There is a deep connection between generators of (stable) subordinators and (classical) fractional derivatives. Fractional derivatives, in particular fractional time derivatives, have recently become important tools to model real-world phenomena. There are important applications in Physics, Chemistry and Biology, see e.g.\ Klages \emph{et al.} \cite{2008_Klages}. On the mathematical side, we refer to the monographs \cite{1993_Samko, 1998_Podlubny, 2010_Diethelm, 2019_Meerschaert} and the papers \cite{2007_Mainardi, 2017_Chen, 2018_HernandezHernandez, 2016_HernandezHernandez, 2018_Sin}.

Let us briefly sketch the rationale of this paper in a classical context. Probably the most general way to introduce fractional derivatives on the half-line is via the \textbf{Weyl--Marchaud fractional derivative}, which is given by
\begin{align}\label{intro-e06}
	\mad{}{+}{\alpha} \phi(x)
	&=\frac{\alpha}{\Gamma(1-\alpha)}\int_{0}^\infty \left(\phi(x)-\phi(x-s)\right) \frac{ds}{s^{1+\alpha}}.
\end{align}
The problem is, however, that this form requires $\phi$ to be a function on $\real$ rather than $(0,\infty)$, and we will overcome this problem by suitable extensions of $\phi|_{(0,\infty)}$, see Section~\ref{bd} for details. Even more striking is the fact that \eqref{intro-e06} can be seen as the formal adjoint of the subordinator generator $\Acal$ with $\mu(dt) = \alpha \Gamma(1-\alpha)^{-1} t^{-1-\alpha}\,dt$. Indeed, a formal calculation in $L^2(\real)$ gives
\begin{align*}
	\Acal^*u(x) = au(x) - b\cdot \frac d{dx}u(x) - \int_{(0,\infty)} \left(u(x)-u(x-t)\right)\mu(dt).
\end{align*}
This indicates that fractional derivatives are closely related to decreasing subordinators (which can be identified with subordinators running backwards in time). In the classical setting, this situation was studied by Du \emph{et al.} \cite{2021_Du}, where the (classical) \textbf{censored fractional derivative} is given by 
\begin{equation}\label{cen-der-002}
	\ced{0}{}{\alpha}\phi(x)=\frac{\alpha}{\Gamma(1-\alpha)}\int_0^x \left(\phi(x)-\phi(x-s)\right) \frac{ds}{s^{\alpha+1}},
	\quad\alpha\in (0, 1).
\end{equation}
Notice that the integral extends over $(0,x)$ rather than $(0, \infty)$, due to ``censoring''. The relation between $\ced{0}{}{\alpha}$ and the censored decreasing stable subordinator is also discussed in \cite{2021_Du}.

With a view to the Weyl--Marchaud form of the fractional derivative, we introduce Bernstein fractional derivatives (of Riemann--Liouville and Caputo type) and we establish the connection with the well-known Bernstein functional calculus, that arises in connection with Bochner's subordination. It turns out that the Bernstein fractional derivatives coincide with various generalizations of fractional derivatives using Sonine pairs (sometimes called convolution-type derivatives, distributed order derivatives etc.), see Kochubei~\cite{2011_Kochubei}, Toaldo~\cite{2015_Toaldo}, Chen~\cite{2017_Chen}, so we do not claim originality here, but the embedding into the existing theory of subordination is new.

In Section~\ref{cvp} we generalize the results by Du \emph{et al.} on censored fractional derivatives, the main result being the inversion of the censored Bernstein derivative by using the Bernstein fractional integral. Section~\ref{re} contains the resolvent equation related to the censored Bernstein derivative, which is needed for the Hille--Yosida approach to the censored decreasing subordinator $S^c$. In Section~5 we give a first (pathwise) construction using piecing out, and in Section~6 we show that $S^c$ as a Feller process whose generator is the censored Bernstein derivative. Here we indicate a further construction via the Hille--Yosida theorem. Finally, we see that $S^c$ has finite life-time.
The appendix contains a short and self-contained approach to Sonine pairs using special Bernstein functions.

\section{Bernstein derivatives}\label{bd}

In this section we give a systematic description of fractional derivatives (of Riemann--Liouville, Caputo and Weyl--Marchaud type) on the half-line which are induced by Bernstein functions. To keep notation simple, we restrict ourselves to derivatives of order $\alpha < 1$.

Recall, see e.g.\ Samko \emph{et.~al.}~\cite[(2.17)]{1993_Samko}, that the classical \textbf{Riemann--Liouville fractional derivative} on $(0,\infty)$ is given by
\begin{gather}\label{bd-e02}
	\rld{0}{}{\alpha}\phi(x) = \frac 1{\Gamma(1-\alpha)} \frac{d}{dx}\int_0^x \frac 1{(x-u)^{\alpha}}\,\phi(u)\,du, \quad x>0,
\intertext{and the \textbf{Caputo fractional derivative} is of the form}\label{bd-e04}
	\cd{0}{}{\alpha}\phi(x) = \frac 1{\Gamma(1-\alpha)}\int_0^x \frac 1{(x-u)^{\alpha}}\,\frac{d}{du}\,\phi(u)\,du, \quad x>0.
\end{gather}
Using integration by parts, the connection between these derivatives turns out to be
\begin{gather}\label{bd-e06}
	\rld{0}{}{\alpha}\phi(x)
	= \cd{0}{}{\alpha}\phi(x) + \frac 1{\Gamma(1-\alpha)}\frac{\phi(0+)}{x^\alpha}, \quad x>0.
\end{gather}
Yet another integration by parts reveals that for $x>0$
\begin{align}\label{bd-e08}
	\rld{0}{}{\alpha}\phi(x)
	&= \frac{\alpha}{\Gamma(1-\alpha)} \int_0^x \left(\phi(x)-\phi(x-u)\right) \frac{du}{u^{1+\alpha}} + \frac 1{\Gamma(1-\alpha)}\frac{\phi(x)}{x^\alpha}\\
	\label{bd-e10}
	&= \frac{\alpha}{\Gamma(1-\alpha)} \int_0^x \left(\phi(x)-\phi(x-u)\right) \frac{du}{u^{1+\alpha}} 
	+ \frac{\alpha}{\Gamma(1-\alpha)} \int_x^\infty \phi(x)\,\frac{du}{u^{1+\alpha}}.
\end{align}
If we extend $\phi:(0,\infty)\to\real$ to $\real$ by setting
\begin{gather}\label{bd-e12}
	\phi^\circ(x) := \Ext_\circ\phi(x) := 
		\begin{cases}
		\phi(x), & x>0,\\
		0, &x\leq 0,
		\end{cases}
		\qquad\text{(\textbf{killing extension})}
\end{gather}
we see that \eqref{bd-e10} becomes
\begin{align}\label{bd-e14}
	\rld{0}{}{\alpha}\phi(x)
	&= \frac{\alpha}{\Gamma(1-\alpha)} \int_0^\infty \left(\phi^\circ(x)-\phi^\circ(x-u)\right) \frac{du}{u^{1+\alpha}}, \quad x>0,
\end{align}
and this is the  \textbf{(Weyl--)Marchaud representation}, often denoted by $\mad{}{+}{\alpha}\phi(x)$. In the same way we can get the Caputo derivative
\begin{align}\label{bd-e16}
	\cd{0}{}{\alpha}\phi(x)
	&= \frac{\alpha}{\Gamma(1-\alpha)} \int_0^\infty \left(\phi^\#(x)-\phi^\#(x-u)\right) \frac{du}{u^{1+\alpha}},\quad x>0,
\end{align}
if we use the following extension of $\phi$ to $\real$: 
\begin{gather}\label{bd-e18}
	\phi^\#(x) := \Ext_\#\phi(x) := 
	\begin{cases}
		\phi(x), & x>0,\\
		\phi(0+), &x\leq 0,
	\end{cases}
	\qquad\text{(\textbf{sticky extension})}.
\end{gather}
The well-definedness of the integrals \eqref{bd-e02}--\eqref{bd-e04} require different smoothness and decay properties of $\phi$. In general, the Marchaud representation (combined with \eqref{bd-e06}) extends both $\rld{0}{}{\alpha}$ and $\cd{0}{}{\alpha}$. 

Using the Weyl--Marchaud derivative it is possible to make the connection between fractional derivatives and Bernstein functions. Recall, cf.\ \cite[p.~vii]{2012_Schilling}, the well known formula
\begin{gather}\label{bd-e20}
	x^\alpha = \frac{\alpha}{\Gamma(1-\alpha)} \int_0^\infty \left(1-e^{-ux}\right) \frac{du}{u^{1+\alpha}}, \quad x>0.
\end{gather}
Introducing the shift operators $T_t\phi(x):=\phi^\circ(x-t)$, $x\in\real$, $t>0$, it is easy to see that $(T_t\phi|_{(0,\infty)})^\circ = T_t\phi^\circ$, and that $(T_t)_{t\geq 0}$ is an operator semigroup. The Laplace \textbf{symbol} of $T_t$ is
\begin{gather*}
	e^{-\lambda t},
	\text{\ \ that is,\ \ }
	\Lscr[T_t\phi;\lambda] = e^{-\lambda t}\Lscr[\phi;\lambda],
\end{gather*}
and the infinitesimal generator of the semigroup $(T_t)_{t\geq 0}$ is $\Acal=-\frac d{dx}$. Thus, \eqref{bd-e10} resp.\ \eqref{bd-e14} becomes
\begin{gather*}
	\rld{0}{}{\alpha}\phi(x) = \left(\frac{d}{dx}\right)^\alpha\phi(x),\quad x>0,
\end{gather*}
where $\left(\frac{d}{dx}\right)^\alpha$ is understood in the sense of the Bernstein functional calculus developed in \cite{1998_Schillinga} and \cite[Ch.~13]{2012_Schilling}. The above equality holds on the intersection of the respective domains. To wit, $-\left(\frac{d}{dx}\right)^\alpha = -f\left(-\Acal\right)$, $f(x)=x^\alpha$, $\Acal=-\frac d{dx}$, is the generator of the Bochner subordination (with a stable subordinator) of the deterministic motion $t\mapsto x-t$.

We can now define the (Marchaud form of the) Bernstein fractional derivative for a given Bernstein function $f$. This derivative extends the above sketched relations between classical fractional derivatives.

\begin{definition}\label{bd-03}
Let $f$ be a Bernstein function with triplet $(0,b,\mu)$ and $u:\real\to\real$ be some function. The \textbf{Bernstein fractional derivative} of $u$ with respect to $f$ is defined by
\begin{align}\label{bd-e22}
    \mad{}{+}{f}u(x)
    := b\cdot \frac{d}{dx}u(x) + \int_{(0,\infty)} \left(u(x)-u(x-t)\right)\mu(dt)
\end{align}
provided that this expression is well-defined.
\end{definition}

At the moment this is a formal definition. If $b>0$, we need to make sure that $u'$ exists in some (weak) sense, defying the idea to have a truly ``fractional'' derivative. So we assume from now on that $b=0$.

We observe that the derivative is defined for functions on the whole real line $\real$, although we are actually interested in functions on the half-line $(0,\infty)$. Since we may always extend $\phi$ from $(0,\infty)$ to the whole line, the definition \eqref{bd-e22} is both more general and more convenient computationally. In analogy with the classical case discussed above, cf.\ \eqref{bd-e08}, \eqref{bd-e10}, we get in this way the Riemann--Liouville form of the Bernstein fractional derivative:
\begin{align}\label{bd-e24}
	\rld{0}{}{f}\phi(x)
	:= \mad{}{+}{f}\phi^\circ(x)
	&= \int_{(0,\infty)} \left(\phi^\circ(x)-\phi^\circ(x-s)\right)\mu(ds)\\
	\notag
	&= \int_{(0,x]}\left(\phi(x)-\phi(x-s)\right)\mu(ds)+\phi(x)\bar\mu(x),
\end{align}    
as long as the integrals are well-defined.  Since $\mad{}{+}{f}\phi^\circ|_{(-\infty,0]}=0$, $\rld{}{+}{f}$ can be seen as an operator acting on functions on the half-line. In fact, there is also the analogue of the classical Riemann--Liouville derivative, if the integral in \eqref{bd-e24} exists:
\begin{align}
	\rld{0}{}{f}\phi(x)
	= \int_{(0,\infty)} \left(\phi^\circ(x)-\phi^\circ(x-s)\right)\mu(ds)
\notag	&= \int_{(0,\infty)} \frac d{dx} \int_{x-s}^x \phi^\circ(t)\,dt\,\mu(ds)\\
\notag	&= \frac d{dx} \int_{(0,\infty)}  \int_{0}^s \phi^\circ(x-t)\,dt\,\mu(ds)\\
\notag	&= \frac d{dx} \int_0^\infty \phi^\circ(x-t) \int_{(t,\infty)}  \mu(ds)\,dt\\
\notag	&= \frac d{dx} \int_0^\infty \phi^\circ(x-t) \bar\mu(t)\,dt\\
\label{bd-e26}	&= \frac d{dx} \int_0^x \phi(x-t)\bar\mu(t)\,dt.
\end{align}

In order to make things rigorous, we consider $\rld{0}{}{f}$ (or $\mad{}{+}{f}\circ\Ext_\circ$) in the space $L^1(0,\infty)$, which we may interpret as subspace of $L^1(\real)$, if we use the killing extension \eqref{bd-e12}. Indeed, if $\mu$ is a finite measure, we have
\begin{align}
	\int_0^\infty\left|\int_{(0,\infty)} \left(\phi^\circ(x)-\phi^\circ(x-t)\right)\mu(dt)\right| dx
\notag	&\leq \int_{(0,\infty)}\int_0^\infty \left|\phi^\circ(x)-\phi^\circ(x-t)\right| dx\, \mu(dt)\\
\label{bd-e27}	&\leq 2\mu(0,\infty)\|\phi\|_{L^1(0,\infty)},
\end{align}
and \eqref{bd-e24} is well-defined and continuous on $L^1(0,\infty)$.

Using the idea of the Yosida approximation, we approximate $f(x)$ by a sequence $f_n(x)$ of Bernstein functions with bounded jump measures:
\begin{gather*}
	f_n(x) := \frac{nf(x)}{n+f(x)} = n\left(1-\frac{n}{n+f(x)}\right) \xrightarrow[n\to\infty]{}f(x),\quad x>0.
\end{gather*}
Since $f$ is a Bernstein function, $\frac{n}{n+f}$ is completely monotone, hence the Laplace transform of a measure $\rho_n$ with total mass $\Lscr[\rho_n;0] = \frac{n}{n+f(0)}=1$. Thus,
\begin{gather}\label{bd-e28}
	f_n(x) = n\int_{(0,\infty)} \left(1-e^{-tx}\right)\rho_n(dt).
\end{gather}
In particular, 
\begin{gather}
	\rld{0}{}{f_n}\phi(x) = \int_{(0,\infty)} \left(\phi^\circ(x) - \phi^\circ(x-t)\right) n\rho_n(dt)
\end{gather}
is a bounded operator on $L^1(0,\infty)$, see \eqref{bd-e27}. Therefore, the following definition gives a natural domain for $\rld{0}{}{f}$ for a general $f$ (if $b=0$):
\begin{gather}\label{bd-e30}
	\Dcal
	:= \left\{\phi\in L^1(0,\infty) \mid \exists \psi\in L^1(0,\infty)\::\: \lim_{n\to\infty}\left\|\rld{0}{}{f_n}\phi - \psi\right\|_{L^1(0,\infty)}=0\right\},
\end{gather}
and we set $\rld{0}{}{f}\phi := \psi$ on $\Dcal\bigl(\rld{0}{}{f}\bigr) := \Dcal$. In Corollary~\ref{bd-07} we will identify $(-\rld{0}{}{f},\Dcal)$ with the infinitesimal generator of a subordinate operator semigroup. If we take (for sufficiently nice functions $\phi$) the Laplace transform in \eqref{bd-e24}, the following lemma shows that the present definition of $\rld{0}{}{f}$ is consistent with  Definition~\ref{bd-03}.
\begin{lemma}\label{bd-05}
	Let $f$ be a Bernstein function with triplet $(0,0,\mu)$, $f_n = nf/(n+f)$ as in \eqref{bd-e28}, and $\phi,\psi\in L^1(0,\infty)$. The following assertions are equivalent
	\begin{enumerate}[label=\upshape\alph*),ref=\upshape\alph*),leftmargin=*,itemsep=4pt,itemindent=\parindent]
		\item\label{bd-05-a} 
			$\Lscr[\psi;\lambda] = f(\lambda)\Lscr[\phi;\lambda]$ for all $\lambda >0$.
		\item\label{bd-05-b}
			$\lim_{n\to\infty}\left\|\rld{0}{}{f_n}\phi - \psi\right\|_{L^1(0,\infty)}=0$.
	\end{enumerate}
\end{lemma}
\begin{proof}
	\ref{bd-05-a}$\Rightarrow$\ref{bd-05-b}:
	A direct computation shows that
	\begin{align*}
		\Lscr\left[\rld{0}{}{f_n}\phi;\lambda\right]
		= f_n(\lambda)\Lscr[\phi;\lambda]
		= \frac{f_n(\lambda)}{f(\lambda)} f(\lambda)\Lscr[\phi;\lambda]
		= \frac{n}{n+f(\lambda)} \Lscr[\psi;\lambda]
		= \Lscr[\rho_n;\lambda] \Lscr[\psi;\lambda].
	\end{align*}
	In view of the convolution theorem and the uniqueness of the Laplace transform we have $\rld{0}{}{f_n}\phi = \psi*\rho_n$. Using the fact that $\rho_n$ is a probability measure, we get
	\begin{align*}
		\left\|\psi*\rho_n - \psi\right\|_{L^1(0,\infty)}
		&= \int_0^\infty \left|\int_{(0,x]} \left(\psi(x-t)-\psi(x)\right)\rho_n(dt) - \psi(x)\int_{(x,\infty)}\rho_n(dt)\right| dt\\
		&\leq \int_0^\infty \int_{(0,x]} \left|\psi(x-t)-\psi(x)\right|\rho_n(dt)\,dx + \int_0^\infty \left|\psi(x)\right|\int_{(x,\infty)}\rho_n(dt)\, dt\\
		&\leq \int_{(0,\infty)}\int_0^\infty \left|\psi(x-t)-\psi(x)\right| dx \,\rho_n(dt) 
		+ \int_{(0,\infty)} \int_0^t \left|\psi(x)\right| dx \, \rho_n(dt).
	\end{align*} 
	From $\Lscr[\rho_n;\lambda] = n/(n+f(\lambda))\to 1$ we conclude that $\rho_n$ converges weakly (in the sense of measures) to $\delta_0$. Since
	the integrands appearing inside the integrals $\int\cdots\rho_n(dt)$ are bounded and continuous functions, which tend to zero as $t\downarrow 0$, we get
	\begin{gather*}
		\lim_{n\to\infty} \left\|\psi*\rho_n - \psi\right\|_{L^1(0,\infty)} = 0.
	\end{gather*}
	
	\medskip\noindent
	\ref{bd-05-b}$\Rightarrow$\ref{bd-05-a}: As in the first part, a direct calculation yields
	\begin{gather*}
		\Lscr\left[\rld{0}{}{f_n}\phi;\lambda\right] = \Lscr[\rho_n;\lambda] \Lscr[\phi;\lambda].
	\end{gather*}
	The $L^1$-convergence of $\rld{0}{}{f_n}\phi$ resp.\ the weak convergence of the measures $\rho_n\to\delta_0$ now imply \ref{bd-05-a}.
\end{proof}

\begin{corollary}\label{bd-07}
	Let $T_t:L^1(0,\infty)\to L^1(0,\infty)$, $T_t\phi(x):=\phi^\circ(x-t)$, $x,t>0$, be the shift semigroup on the half-line and denote by $-f\left(\frac d{dx}\right)$ the generator of the semigroup that arises via Bochner's subordination of the shift semigroup. Then $\Dcal\left(\rld{0}{}{f}\right)$ is its domain and $f\left(\frac d{dx}\right) = \rld{0}{}{f}$.
\end{corollary}
\begin{proof}
	If we embed $L^1(0,\infty)$ into $L^1(\real)$ by $\phi \mapsto u:=\Ext_\circ\phi$, the claim follows from the general theory of the Bernstein functional calculus, cf.\ \cite[Cor.~13.20, Rem.~13.21]{2012_Schilling}. 
\end{proof}

The inverse of the classical Riemann--Liouville fractional derivative is the Riemann--Liouville fractional integral. A similar result holds for Bernstein fractional derivatives, but the corresponding fractional integral is more complicated. Its existence is guaranteed by so-called Sonine pairs. We will discuss this part of the theory in the context of continuous functions (as we need it later in this form), which will also give a stronger $L^1$-version in weighted function spaces. Sonine pairs have been studied by various authors, see e.g.\ \cite{2011_Kochubei, 2002_Samko, 2003_Samko, 2020_Hanyga, 2022_Luchko}; a self-contained exposition of Sonine pairs and special Bernstein functions is given in the Appendix.
\begin{definition}\label{bd-09}  
	Let $g,h:(0,\infty)\to[0,\infty)$ be positive measurable functions such that $g,h\in L^1_{\mathrm{loc}}[0,\infty)$. We call $(g,h)$ a \textbf{positive Sonine pair} if 
	\begin{equation}\label{bd-e32}
		h*g(x) = \int_{(0,x)} h(x-t)\,g(t)\,dt \equiv 1, \quad x>0.
	\end{equation}
	In abuse of notation, we still call $(g,h)$ a positive Sonine pair, if either $g$ or $h$ is a positive measure on $(0, \infty)$. 
\end{definition}
If $(g,h)$ is a positive Sonine pair, the Laplace transform satisfies
\begin{gather*}
	\Lscr[\I;\lambda]
	=\int_{0}^\infty e^{-\lambda z}\,dz
	= \frac{1}{\lambda}
	=\Lscr[h\ast g;\lambda]
	= \Lscr[h;\lambda]\cdot\Lscr[g;\lambda].
\end{gather*}
Hence, we obtain the identity
\begin{gather*}
	\Lscr[h;\lambda]
	= \frac{1}{\lambda \Lscr[g;\lambda]},
\end{gather*}
which leads to a necessary and sufficient condition for the existence of a positive Sonine pair in terms of completely monotone functions, see the Appendix.

From now on we will make some assumptions on the Bernstein function $f$.
\begin{assumption}
\begin{enumerate}[label=\upshape\bfseries A\arabic*.,ref=\upshape A\arabic*,leftmargin=*,labelindent=\parindent]
	\item\label{A1}
	$f$ is a Bernstein function satisfying
	\begin{gather*}
		f(0+) = \lim_{x\to 0} f(x) = 0,\quad
		\lim_{x\to\infty} f(x) = +\infty,\quad
		f'(0+) = \lim_{x\to 0} \frac{f(x)}{x} = +\infty,\quad\
		\lim_{x\to\infty} \frac{f(x)}{x} = 0.
	\end{gather*}
	\item\label{A2}
	$f$ is a \textbf{complete Bernstein function}, i.e.\  the L\'{e}vy measure $\mu$ has a completely monotone density $m$ with respect to Lebesgue measure.  
\end{enumerate}
\end{assumption}
It is not difficult to see that the conditions in Assumption~\ref{A1} can be equivalently expressed as: The L\'evy triplet of $f$ is of the form $(0,0,\mu)$ and the L\'evy measure satisfies $\int_{(0,1)} d\mu = +\infty$ and $\int_{(1,\infty)} x\,\mu(dx)=+\infty$, see e.g.\ \cite[Ch.~3]{2012_Schilling}. Moreover, it is obvious from~\ref{A1} that $f$ satisfies \ref{A1} if, and only if, the conjugate function $f^*(x) = \frac x{f(x)}$ is a Bernstein function that satisfies \ref{A1}.	Assumption \ref{A2} guarantees that $f^*$ is a complete Bernstein function, hence a Bernstein function, cf.\ \cite[Prop.~7.1]{2012_Schilling}.

The following lemma is a simple consequence of a more general result on special Bernstein functions, which we will defer to the appendix.
\begin{lemma}\label{bd-11}
	Let $f$ be a Bernstein function satisfying Assumptions~\ref{A1} and~\ref{A2}, and let $\bar\mu(x) = \mu(x,\infty)$, $x>0$, be the tail of the L\'evy measure.
	\begin{enumerate}[label=\upshape\alph*),ref=\upshape\alph*),leftmargin=*,itemsep=4pt,itemindent=\parindent]
	\item\label{bd-11-a} 
	There exists a function $k\in C(0,\infty)$ such that $(\bar\mu, k)$ is a positive Sonine pair. 
	
	\item\label{bd-11-b} 
	$k$ is the tail function $\bar\mu^*(x) = \mu^*(x,\infty)$ of the L\'evy measure $\mu^*$ of the conjugate Bernstein function $f^*(x) =  x/{f(x)}$ and $\Lscr[k;x] = 1/{f(x)}$.
	
	\item\label{bd-11-c} 
	$\bar\mu$ and $k$ are completely monotone functions.
	\end{enumerate}
\end{lemma}

\begin{definition}\label{bd-13}
	Let $\phi:(0,\infty)\to\real$ be a function, $f$ a Bernstein function satisfying \ref{A1} and \ref{A2} such that $(\bar\mu,k)$ is a positive Sonine pair. 
	The \textbf{Bernstein--Riemann--Liouville integral} with respect to the Bernstein function $f$ is defined by
	\begin{align}\label{bd-e36}
		\rli{0}{}{f}\phi(x)
		:= \int_0^x \phi(x-z)k(z)\,dz,\quad x>0,
	\end{align}
	provided that the integral is well-defined.
\end{definition}

In order to deal with well-definedness of the Bernstein fractional integral in spaces of continuous functions we introduce some function spaces.
\begin{definition}\label{bd-15}
Let $T>0$. We define the function space
\begin{align*}
	C_{\bar{\mu}}(0,T)
	&:=\left\{\phi \in C (0, T)\cap L^1(0,T) \mid \bar{\mu}*\phi\in C^1(0,T)\right\},\\
	C_{\bar{\mu}}[0,T)&:=C[0,T)\cap C_{\bar{\mu}}(0,T).
\end{align*}
\end{definition}
It is immediately clear from the definition of $C_{\bar{\mu}}(0,T)$ that 
	$
	\rld{0}{}{f}:C_{\bar{\mu}}(0,T)\to C(0,T)
	$.

\begin{theorem}\label{bd-17} 
	If $\psi\in C(0,T)\cap L^1(0,T)$, then $\rli{0}{}{f}\psi \in C_{\bar{\mu}}(0,T)$, and $\rld{0}{}{f}$ is the left inverse of $\rli{0}{}{f}$.
\end{theorem}
\begin{proof} 
	We need to show $\rli{0}{}{f} \psi  \in C(0, T)\cap L^1(0,T)$ and $\bar{\mu}*\left[\rli{0}{}{f}\psi\right]\in C^1(0,T)$. Since  $\psi$ and $k$ are in $C(0,T)\cap L^1(0,T)$, $\rli{0}{}{f}\psi(x)$ is well-defined and finite for all $x\in (0,T)$. To show continuity, we fix $T> T_0 > 0$ and $x\in (T_0,T)$. For $\epsilon \in (0, T_0/4)$ we have
\begin{align*}
	&\left| \rli{0}{}{f}\psi(x-\epsilon) - \rli{0}{}{f}\psi(x) \right|\\
	&\quad= \left|\int_{x-\epsilon}^x \psi(s)k(x-s)\,ds 
		+ \int_{0}^{x-\epsilon} \psi(s) \left(k(x-s)-k(x-\epsilon-s)\right) ds\right| \\
	&\quad\leq \int_{x-\epsilon}^x \left| \psi(s) k(x-s) \right| ds 
	+ \int_{0}^{x/2} \left|\psi(s)\left(k(x-s)-k(x-\epsilon-s)\right)\right| ds\\
	&\qquad\mbox{} + \int_{x/2}^{x-\epsilon} \left|\psi(s)\left(k(x-s)-k(x-\epsilon-s)\right)\right| ds\\
	&\quad\leq \|\psi\|_{C(T_0-\epsilon, T)} \int_{x-\epsilon}^x k(x-s) \,ds
	+ \left\|k(\cdot)-k(\cdot-\epsilon)\right\|_{C(T_0/2,T)}\int_0^T |\psi(s)|\,ds\\
	&\qquad\mbox{}+ \|\psi\|_{C(T_0/2,T)} \int_0^T \left|k(\epsilon+s)-k(s)\right| ds.
\end{align*}
By the continuity of $k$ on $(T_0/4,T)$, and the fact that $k\in L^1(0,T)$, we see that the right-hand side tends to $0$ as $\epsilon\to 0$. Together with a similar calculation with $x$ and $x+\epsilon$ we get $\rli{0}{}{f}\psi\in C(T_0,T)$ for all $T_0>0$, so we have $\rli{0}{}{f}\psi\in C(0,T)$. The integrability of $\rli{0}{}{f}\psi$ follows from
\begin{equation}\label{bd-e38}\begin{aligned}
	\int_0^T \left|\rli{0}{}{f}\psi(x)\right| dx 
	&\leq \int_0^T \int_0^x \left|\psi(r)\right| k(x-r)\,dr\,dx\\
	&=\int_0^T \left|\psi(r)\right| \int_r^T  k(x-r) \,dx \,dr\\
	&\leq \int_0^T  k(r) \,dr \int_0^T \left|\psi(r)\right| dr
	<\infty.
\end{aligned}\end{equation}
As $\rli{0}{}{f}\psi \in C(0, T)\cap L^1(0, T)$, $\bar{\mu}*\rli{0}{}{f}\psi$ is well-defined on $(0, T)$. We calculate for $x\in (0, T)$,
\begin{align*}
	\bar{\mu}*\rli{0}{}{f}\psi(x)
	&=\int_0^x\bar{\mu}(x-r)\rli{0}{}{f}\psi(r)\,dr \\
	&=\int_0^x\psi(s)\int_s^x \bar{\mu}(x-r)k(r-s)\,dr\,ds\\
	&=\int_0^x\psi(s)\,ds,
\end{align*}
which proves $\bar{\mu}*\left(\rli{0}{}{f}\psi\right) \in C^1(0,T)$ and $\frac{d}{dx}\left(\bar{\mu}*\rli{0}{}{f}\psi\right)\equiv\psi$ on $(0,T)$.  Because of \eqref{bd-e26} we have that $ \rld{0}{}{f}$ is the left-inverse of $\rli{0}{}{f}$.
\end{proof}
As a consequence, we obtain the following theorem.
\begin{theorem} \label{bd-19}
	Let $f$ be a Bernstein function satisfying Assumptions~\ref{A1} and~\ref{A2}. 
	\begin{enumerate}[label=\upshape\alph*),ref=\upshape\alph*),leftmargin=*]\itemsep=3pt
	\item\label{bd-19-a}
	Let $\psi \in C(0, T)\cap L^1(0, T)$ for some $T>0$. The following assertions are equivalent:
	\begin{enumerate}[label=\upshape(\roman*),ref=\upshape(\roman*)]\itemsep=3pt
	\item \label{bd-19-i}
		$\rli{0}{}{f}\psi = \phi$.
	\item\label{bd-19-ii}
		$\rld{0}{}{f}\phi = \psi$ for some $\phi\in C_{\bar{\mu}}(0,T)$ such that $\lim_{x\to 0} (\bar{\mu}*\phi)(x)=0$.
	\end{enumerate}
	If $\phi  \in C_{\bar{\mu}}[0, T)$ and $\rld{0}{}{f}\phi=0$, then $\phi=0$.
	
	\item\label{bd-19-b}
	Assume that \ref{bd-19-i} or \ref{bd-19-ii} hold for all $T>0$, 
	and $\sup_{t\geq 1} \left|\phi(t) e^{-\epsilon t}\right|\leq c_\epsilon<\infty$ for every $\epsilon>0$, then \ref{bd-19-i}, \ref{bd-19-ii} are also equivalent to
	\begin{enumerate}[label=\upshape(\roman*),ref=\upshape(\roman*)]\setcounter{enumii}{2}\itemsep=3pt
	\item\label{bd-19-iii}
		$\Lscr[\rld{0}{}{f}\phi;\lambda] = f(\lambda)\Lscr[\phi;\lambda]$ for all $\lambda > 0$.
	\end{enumerate}
	\end{enumerate}
\end{theorem}
\begin{proof}
\ref{bd-19-a}\ \ 
	\ref{bd-19-i}$\Rightarrow$\ref{bd-19-ii}: From Theorem \ref{bd-17} we see that $\phi = \rli{0}{}{f}\psi\in C_{\bar{\mu}}(0,T)$ and $\rld{0}{}{f}\phi=\psi$. Moreover, using \eqref{bd-e26},
	\begin{align*}
		\bar{\mu}\ast\phi(x)
		= \bar{\mu}\ast \rli{0}{}{f}\psi(x)
		= \int_0^x\psi(z)\,dz
		\xrightarrow[x\to 0]{} 0.
	\end{align*} 
	
	\smallskip\noindent
	\ref{bd-19-ii}$\Rightarrow$\ref{bd-19-i}:
	Using the results of Theorem \ref{bd-17}, $ \rld{0}{}{f}$ is the  left inverse of $\rli{0}{}{f}$. Thus, 
	\begin{gather*}
		\rld{0}{}{f}\phi(x) = \psi(x) = \rld{0}{}{f}\!\rli{0}{}{f}\psi(x).
	\end{gather*}  
	Consequently,  $\rld{0}{}{f}\left(\phi-\rli{0}{}{f}\psi\right)=0$ and, in view of \eqref{bd-e26},  $\bar{\mu}*\left(\phi-\rli{0}{}{f}\psi\right)$ is constant in $(0,T)$.  

	The left-inverse property and \eqref{bd-e26} show $\lim_{x\to 0}\bar\mu*\rli{0}{}{f}\psi(x) = \lim_{x\to 0}\int_0^x \psi(z)\,dz = 0$ and, by assumption,  $\lim_{x\to0} (\bar{\mu}*\phi)(x)=0$. Therefore, $\bar{\mu}*\left(\phi-\rli{0}{}{f}\psi\right) \equiv 0$.
	
	Using once again Theorem \ref{bd-17} and \eqref{bd-e26}, we get 
	\begin{align*}
		\phi-\rli{0}{}{f}\psi
		&=\rld{0}{}{f}\rli{0}{}{f}\left[\phi-\rli{0}{}{f}\psi\right]\\
		&=\frac{d}{dx}\left[\bar\mu*\left(\rli{0}{}{f}\left(\phi-\rli{0}{}{f}\psi\right)\right)\right]\\
		&=\frac{d}{dx}\left[k*\left(\bar\mu*\left(\phi-\rli{0}{}{f}\psi\right)\right)\right]\\
		&=\frac{d}{dx}[k*0]
		\:=\:0,
	\end{align*}
	If $\phi\in C_{\bar{\mu}}[0, T)$, then $\left|\bar{\mu}*\phi(x)\right| \leq \|\phi\|_{C[0,T)} \int_0^x\bar{\mu}(s)\,ds\xrightarrow[x\to 0]{} 0$.
	Applying \ref{bd-19-i} we obtain $\phi=\rli{0}{}{f}0=0$.

\medskip\noindent\ref{bd-19-b}\ \ 
	Assume now that $\phi$ does not grow exponentially as $t\to\infty$.
	
	\ref{bd-19-ii}$\Rightarrow$\ref{bd-19-iii}:
	Let \ref{bd-19-ii} hold for all $T>0$. Using \eqref{bd-e26}, integration by parts and Fubini's theorem, we get for fixed $\lambda > 0$
	\begin{align*}
		\Lscr_T\left[\rld{0}{}{f}\phi;\lambda\right]
		:=&\: \int_0^T \frac d{dx}\int_0^x \phi(x-t)\bar\mu(t)\,dt \: e^{-\lambda x}\,dx\\
		=&\:  e^{-\lambda T}\int_0^T \phi(T-t)\bar\mu(t)\,dt + \int_0^T\int_0^x \phi(x-t)\bar\mu(t)\,dt\: \lambda e^{-\lambda x}\,dx\\
		=&\:  e^{-\lambda T}\int_0^T \phi(T-t)\bar\mu(t)\,dt + \int_0^T \lambda \bar\mu(t) \int_t^T e^{-\lambda x}\phi(x-t)\,dx \,dt\\
		=&\:  e^{-\lambda T}\int_0^T \phi(T-t)\bar\mu(t)\,dt + \int_0^T \lambda e^{-\lambda t}\bar\mu(t) \int_0^{T-t}  e^{-\lambda x}\phi(x)\,dx\,dt.
	\end{align*}
	Since $\phi$ is integrable over $(0,1)$ and grows slower than $e^{\frac 12\lambda t}$ as $t\to\infty$, it is not hard to see that the Laplace transform $\int_0^{\infty} e^{-\lambda x} |\phi(x)|\,dx$ is finite.\footnote{Using Wiener's Tauberian theorem one can, with some effort, remove the exponential growth assumption on $\phi$ and show that the finiteness of the Laplace transform, $\int_0^\infty e^{-\lambda x}|\phi(x)|\,dx < \infty$, ensures that $\lim_{T\to\infty} e^{-\lambda T}\int_0^T \phi(T-t)\bar\mu(t)\,dt = 0$.} Moreover, the first integral expression
	\begin{align*}
		\left|e^{-\lambda T} \int_0^T \phi(T-t)\bar\mu(t)\,dt\right|
		&\leq C_{\lambda} e^{-\frac 12\lambda T} \int_0^\infty e^{-\frac 12 \lambda t} \bar\mu(t)\,dt\\
		&= C_{\lambda} e^{-\frac 12\lambda T} \int_0^\infty e^{-\frac 12 \lambda t} \int_{(t,\infty)}\mu(ds)\,dt\\
		&= C_{\lambda} e^{-\frac 12\lambda T} \int_{(0,\infty)} \int_0^s e^{-\frac 12 \lambda t}\,dt \,\mu(ds)\\
		&=  C_{\lambda} e^{-\frac 12\lambda T} \frac 2\lambda\int_{(0,\infty)} \left(1-e^{-\frac 12\lambda s}\right)\mu(ds)\\
		&=  C_{\lambda} e^{-\frac 12\lambda T} \frac 2\lambda f\left(\tfrac 12\lambda\right)
	\end{align*}
	tends to zero as $T\to\infty$. This allows us to let $T\to\infty$ and, with almost the same calculation as for the first integral, we get
	\begin{align*}
		\Lscr\left[\rld{0}{}{f}\phi;\lambda\right]
		= \int_0^\infty \lambda e^{-\lambda t}\bar\mu(t)\,dt \int_0^{\infty} e^{-\lambda x}\phi(x)\,dx
		= f(\lambda) \Lscr[\phi;\lambda].
	\end{align*}
	This finishes the proof of \ref{bd-19-iii}.

	Conversely, assume that \ref{bd-19-iii} holds. In order to show \ref{bd-19-ii}, we have to solve $\rld{0}{}{f}\phi = \psi$ for a given, Laplace-transformable $\psi$. Using Lemma~\ref{bd-11} we define $\phi$ by 
	\begin{gather*}
		\Lscr[\phi;\lambda] := f(\lambda)^{-1} \Lscr[\psi;\lambda] = \Lscr[k;\lambda]\Lscr[\psi;\lambda] = \Lscr[k*\psi;\lambda] = \Lscr\left[\rli{0}{}{f}\psi\right].
	\end{gather*}
	By the uniqueness of the Laplace transform we get $\phi = \rli{0}{}{f}\psi$ and we conclude just as in the proof of \ref{bd-19-i}$\Rightarrow$\ref{bd-19-ii}, that $\phi\in C_{\bar\mu}(0,T)$ and $\lim_{x\to 0}\bar\mu*\phi(x)=0$.
\end{proof}

\section{Censored initial value problem}\label{cvp}
We will now define the censored Bernstein fractional derivative, which is the primary focus of our study. Throughout this section $(\bar\mu,k)$ is a Sonine pair where $\mu$ is the jump measure of a Bernstein function $f$, which satisfies the Assumptions~\ref{A1} and~\ref{A2}.
\begin{definition}\label{cvp-03}
	The \textbf{censored Bernstein derivative} of a function $\phi\in C_{\bar{\mu}}(0,T]$ is 
	\begin{equation}\label{cvp-e02}
		\ced{0}{}{f}\phi(x)
		= \rld{0}{}{f}\phi(x)-\phi(x)\bar{\mu}(x), \quad x\in (0,T].
	\end{equation}
\end{definition}

\begin{remark}\label{cvp-05}
For $\phi \in C_{\bar{\mu}}(0,T]$ we have the following alternative representation of the censored Bernstein derivative:
\begin{align*}
	\ced{0}{}{f}\phi(x)
	=\int_{[0,x]} \left(\phi(x)-\phi(x-s)\right)\mu(ds).
\end{align*}
This representation also explains the name ``censored'' derivative: the jump measure is restricted to $[0,x]$, all values larger than $x$ are disallowed, hence, ``censored''.
\end{remark}

The censored Bernstein derivative is the generalization of the censored fractional derivative, see \cite{2021_Du}. In this section we will construct the inverse of the censored fractional derivative, which is more complicated than the inverse Bernstein derivatives in Section~\ref{bd}. For this we need a few preparations.

\begin{definition}\label{cvp-07}
Let $0<r<x<T$ and $i\in \nat$. The kernels $k_i(x, r)$ are defined by
\begin{align} \label{cvp-e04}
	k_i(x,r)
	&:= \begin{cases}
		\bar{\mu}(r)k(x-r), & i = 1;\\
		\int_r^xk_1(x,s)k_{i-1}(s, r)\,ds, &i \geq 2.
	\end{cases}
\intertext{The corresponding operators are} 
\label{cvp-e06}
\Kcal\phi(x)
&:=\begin{cases} 
	\int_0^xk_1(x,r)\phi(r)\,dr, &x>0;\\
	\phi(0),  &x=0.
\end{cases}\\
\notag
\Kcal^i\phi(x) &:= \underbrace{\Kcal\circ\dots\circ\Kcal}_{i~\text{times}} \phi(x),\quad i\geq 2.
\end{align}
\end{definition}
Observe that $\Kcal\phi|_{(0,\infty)} = \rli{0}{}{f}[\bar{\mu}\phi]|_{(0,\infty)}$ and $\Kcal^i \phi(x) = \int_0^x k_i(x,r)\phi(r)\,dr$ for $x>0$.
It is not difficult to see that $r\mapsto k_i(x,r)$ is a probability density and $\Kcal$ is a linear, positivity preserving operator, i.e.\  $\Kcal \phi\geq0$ if $\phi\geq0$. We need the operators $\Kcal^i$ to construct the inverse of $\ced{0}{}{f}$.
\begin{theorem}\label{cvp-09} 
	Let $(\bar\mu,k)$ be a positive Sonine pair and $\Lscr(\bar\mu; \lambda)=f(\lambda)/\lambda$, $\Lscr(k; \lambda)=1/f(\lambda)$, where $f$ is a Bernstein function satisfying~\ref{A1} and~\ref{A2}. If $\limsup_{x\to 0}\bar{\mu}(x)\int_0^x k(s)\,ds<1$, then there exists for every $\phi_0\in \real$, $T>0$ and $g\in C[0,T]$ a unique function $\phi \in C_{\bar{\mu}}[0,T]$ such that
	\begin{equation}\label{cvp-e08}
	\left\{\begin{aligned}
		\ced{0}{}{f}\phi(x)&= g(x), &&x>0;\\
		\phi(0)&=\phi_0.
	\end{aligned}\right.
	\end{equation}	
	The function $\phi$ is given by the following series representation 
	\begin{equation}\label{cvp-e10}
		\phi(x)-\phi_0
		= \cei{0}{}{f}g(x)
		= \sum_{i=0}^\infty \Kcal^i\left[\rli{0}{}{f} g\right](x).
	\end{equation}
\end{theorem}

For the proof of Theorem \ref{cvp-09} we need a few lemmas. Unless otherwise mentioned, $(\bar\mu,k)$ and $f$ are as in the statement of Theorem~\ref{cvp-09}.
\begin{lemma}\label{cvp-11} 
	Let $T>0$. The limit $q'=\limsup_{x\to 0}\bar{\mu}(x)\int_0^x k(s)\,ds=1$ if, and only if, the supremum $q=\sup_{x\in [0,T]}\bar{\mu}(x)\int_0^x k(s)\,ds=1$.
\end{lemma}
\begin{proof} 
	Since $\bar\mu$ is decreasing, we have
	\begin{align*}
		\bar{\mu}(x)\int_0^x k(s)\,ds
		=\bar{\mu}(x)\int_0^x k(x-s)\,ds
		\leq \int_0^x \bar{\mu}(s)k(x-s)\,ds
		=1.
	\end{align*}
	This shows that $q'\leq q\leq 1$, hence $q'=1$ implies $q=1$. Assume now that $q=1$ and that the supremum is attained at some $x_0\in (0,T]$, i.e.\ $\bar{\mu}(x_0)\int_0^{x_0} k(s)\,ds=1$. Since  $\int_0^{x_0}\bar{\mu}(s) k(x_0-s)\,ds=1$, we have 
	\begin{align*}
	 	0
	 	&=\bar{\mu}(x_0)\int_0^{x_0} k(s)\,ds-\int_0^{x_0}\bar{\mu}(s) k(x_0-s)\,ds\\
		&= \int_0^{x_0} \bar{\mu}(x_0)k(x_0-s)\,ds-\int_0^{x_0}\bar{\mu}(s) k(x_0-s)\,ds\\
	 	&= \int_0^{x_0} (\bar{\mu}(x_0)-\bar{\mu}(s)) k(x_0-s)\,ds.
	\end{align*}
	This  means that  $\bar{\mu}|_{[0,x_0]}$ is constant, since $\bar\mu$ is decreasing. From the Sonine equation we see that $\int_0^xk(s)\,ds$ is constant on $[0, x_0]$.  Therefore,  $k=0$ on $[0,x_0]$. This is a contradiction to $(\bar{\mu}, k)$ being a Sonine pair. Consequently, $\lim_{\epsilon\to 0} \sup_{x\in [\epsilon, T]} \bar\mu(x)\int_0^x k(x-s)\,ds = 1$.
\end{proof}

\begin{lemma}\label{cvp-13} 
	For all $\phi\in C[0,T]$ the inequality $\left|\rli{0}{}{f}\phi(x)\right|\leq\|\phi\|_{C[0,T]}\int_0^x k(s)\,ds$ holds; in particular, $\rli{0}{}{f}$ maps $C[0,T]$ into $C[0,T]$.
\end{lemma}
\begin{proof}
	We have
	\begin{gather*}
		\left|\rli{0}{}{f}\phi(x)\right|
		=\left|\int_0^x \phi(s) k(x-s)\,ds\right|
		\leq \|\phi\|_{C[0,T]}\int_0^x k(s)\,ds.
	\end{gather*}
	This shows that $\rli{0}{}{f}\phi$ is continuous at $x=0$. From Theorem~\ref{bd-17} we conclude that $\rli{0}{}{f}\phi\in C[0,T]$.	
\end{proof}

\begin{lemma}\label{cvp-15}
	The operator $\Kcal$ maps $C[0,T]$ into $C[0,T]$.
\end{lemma}
\begin{proof}
	Let $\phi\in C[0,T]$ and $x>0$. We have
	\begin{align*}
	|\Kcal\phi(x)-\phi(0)|
	&=\left|\int_0^x\bar{\mu}(r) k(x-r) \left(\phi(r)-\phi(0)\right) dr\right|\\
	&= \int_0^x\bar{\mu}(r) k(x-r) \,dr \sup_{r\in [0,x]}\left|\phi(r)-\phi(0)\right|
	= \sup_{r\in [0,x]}\left|\phi(r)-\phi(0)\right|,
\end{align*}
which proves that $\lim_{x\to 0}\Kcal\phi(x)=\Kcal\phi(0)=\phi(0)$, i.e.\ $x\mapsto \Kcal\phi(x)$ is continuous at $x=0$.

Fix $\delta>0$ and let $x,y\in [\delta, T]$ and $\epsilon<\delta$. Without loss of generality we assume that $x>y$. Then
\begin{align*}
	&|\Kcal\phi(x)-\Kcal\phi(y)|
	= \left|\int_0^x\bar{\mu}(r) k(x-r) \phi(r)\,dr-\int_0^y \bar{\mu}(r) k(y-r)\phi(r)\,dr\right|\\
	&\quad\leq \left|\int_0^y\bar{\mu}(r) k(x-r) \phi(r)\,dr - \int_0^y \bar{\mu}(r) k(y-r)\phi(r)\,dr\right|
	+ \left|\int_y^x\bar{\mu}(r) k(x-r) \phi(r)\,dr\right|\\
	&\quad\leq \int_0^y\bar{\mu}(r) |k(x-r)-k(y-r)| |\phi(r)|\,dr+\bar{\mu}(y) \|\phi\|_{\infty} \int_0^{x-y}k(r) \,dr\\
	&\quad\leq 2\|\phi\|_{C[0,T]} \bar{\mu}(y-\epsilon)\int_{y-\epsilon}^y k(y-r)dr+\|\phi\|_{C[0,T]}\int_0^{y-\epsilon}\bar{\mu}(r) |k(x-r)-k(y-r)| dr\\
	&\qquad\mbox{}+\bar{\mu}(y)\|\phi\|_{C[0,T]}\int_0^{x-y}k(r) \,dr\\
	&\quad\leq\|\phi\|_{C[0,T]}\bar{\mu}(\delta-\epsilon)\left(3\int_{0}^{(x-y)\vee \epsilon} k(r)dr+\int_0^{y-\epsilon}\bar{\mu}(r) |k(x-r)-k(y-r)| dr\right)
\end{align*}
The first integral tends to $\int_{0}^{\epsilon} k(r)\,dr$ as $x-y\to 0$. Using dominated convergence with the integrable majorant $2\mu(\cdot)k(y_0-\cdot)$ (notice that $k$ is decreasing and $y_0<y<x$), we see that the second integral vanishes as $x-y\to 0$. Finally, using that $k$ is locally integrable, the whole expression tends to $0$ as $\epsilon\to 0$ and the continuity is proven.
\end{proof}

\begin{lemma}\label{cvp-17}
	Let $q := \sup_{x\in [0,T]}\bar{\mu}(x)\int_0^x k(s)\,ds$. For every $\phi \in C[0,T]$ satisfying $|\phi(x)|\leq M\int_0^x k(s)\,ds$ for all $x\in [0,T]$ and some constant $M>0$ one has $\Kcal\phi(x)\leq M q \int_0^x k(s)\,ds$ for all $x\in [0,T]$. Furthermore, $|\Kcal^i \phi(x)|\leq M q^i \int_0^x k(s)\,ds$, $x\in [0,T]$.
\end{lemma}
\begin{proof} 
	Since $\bar\mu$ is decreasing, we have for all $x>0$ 
	\begin{align*}
		|\Kcal\phi(x)|
		&=\left|\int_0^x \bar{\mu}(r) k(x-r) \phi(r)\,dr\right|\\
		&\leq  \sup_{r\leq x} \left[\phi(r) \bar{\mu}(r)\right] \int_0^x  k(s)\,ds\\
		&\leq M  \sup_{r\leq x} \left[\bar{\mu}(r)\int_0^r  k(s)\,ds\right] \int_0^x k(s)\,ds\\
		&\leq M q \int_0^x k(s)\,ds,
	\end{align*}
	where we use that $|\phi(r)|\leq M\int_0^r k(s)\,ds$. Iterating the above estimate we get for $i\geq 2$
	\begin{align*}
		|\Kcal^i \phi(x)|
		&= \left|\Kcal^{i-1} \Kcal\int_0^x\bar{\mu}(r) k(x-r) \phi(r)\,dr\right|\\
		&\leq M q \Kcal^{i-1} \int_0^x k(r) \,dr\\
		&\leq M q^2 \Kcal^{i-2} \int_0^x k(r) \,dr\\
		&\leq\dots\leq M q^i \int_0^x k(r) \,dr. 
		\qedhere
	\end{align*}
\end{proof}

\begin{corollary}\label{cvp-19} 
	If $\limsup_{x\to 0}\bar{\mu}(x)\int_0^x k(s)\,ds<1$, then $q:=\sup_{x\in [0,T]}\bar{\mu}(x)\int_0^x k(s)\,ds<1$. Let $\phi \in C[0,T]$ such that $|\phi(x)|\leq M\int_0^x k(s)\,ds$ for all $x\in [0,T]$ and some constant $M>0$. Then 
	\begin{gather*} 
		\left|\sum_{i=1}^\infty \Kcal^i\phi (x)\right|
		\leq  \sum_{i=1}^\infty |\Kcal^i \phi(x)|
		\leq M \sum_{i=1}^\infty q^i  \int_0^x k(s)\,ds,
		\quad x\in [0,T],
	\end{gather*} 
	i.e.\ the series converges uniformly; in particular, $\sum_{i=1}^\infty \Kcal^i  \phi \in C[0,T]$.
\end{corollary}

\begin{lemma}\label{cvp-21}  
	Let $g\in C[0,T]$. Under the conditions of Theorem \ref{cvp-09}, the representation \eqref{cvp-e10} of $\cei{0}{}{f}g$ has the following alternative form:
	\begin{align*}
		\cei{0}{}{f}g 
		=\rli{0}{}{f}\left[\bar{\mu} \sum_{i=0}^{\infty}\Kcal^i\left[\bar{\mu}^{-1}g \right]\right]
		=\sum_{i=0}^{\infty}\Kcal^{i+1}\left[\bar{\mu}^{-1}g\right].
	\end{align*}
\end{lemma}
\begin{proof}
	From the definition of the operator $\Kcal$  we obtain
	\begin{align*}
		\cei{0}{}{f}g(x) 
		=\sum_{i=0}^{\infty}\Kcal^i\left[\rli{0}{}{f}g\right](x)
		&=\sum_{i=0}^{\infty}\int_0^xk_i(x,r)\rli{0}{}{f}g(r)\,dr\\
		&=\sum_{i=0}^{\infty}\int_0^xk_i(x,r)\Kcal\left[\bar{\mu}^{-1}g\right](r)\, dr\\
		&=\sum_{i=0}^{\infty}\Kcal^{i+1}\left[\bar{\mu}^{-1}g\right](x)\\
		&=\sum_{i=0}^{\infty}\Kcal\left[\bar{\mu}^{-1} \bar{\mu}\Kcal^i\left[\bar{\mu}^{-1}g\right]\right](x)\\
		&=\sum_{i=0}^{\infty}\rli{0}{}{f}\left[\bar{\mu}\Kcal^i\left[\bar{\mu}^{-1}g \right]\right](x)\\
		&=\rli{0}{}{f}\left[\sum_{i=0}^{\infty}\bar{\mu} \Kcal^i\left[\bar{\mu}^{-1}g \right]\right](x).
		\qedhere
	\end{align*}
\end{proof}

\noindent
After these preparations, we can finally turn to the
\begin{proof}[Proof of Theorem \ref{cvp-09}]
	\emph{Step~1.} We show the existence of $\phi$ using a Picard iteration scheme. Define the following sequence for all $x\in [0,T]$
\begin{equation}\label{cvp-e12}
	\left\{
	\begin{aligned}
	\bar{\phi}_{n+1}(x)&=\rli{0}{}{f}\left[g+\bar{\phi}_n\bar{\mu}\right](x)\\
	\bar{\phi}_{0}(x)&=0.
	\end{aligned}
	\right.
\end{equation}
Thus, observing that $\Kcal\phi = \rli{0}{}{f}\left[\bar\mu\phi\right]$, we get
\begin{gather*}
	\bar{\phi}_{n+1}(x) 
	= \sum_{i=0}^n\Kcal^i\left[\rli{0}{}{f}g(x)\right].
\end{gather*}
For every $x$ the limit $\bar{\phi}(x):=\lim_{n\to \infty}\bar{\phi}_{n+1}(x)=\sum_{i=0}^{\infty}\Kcal^i\left[\rli{0}{}{f}g\right](x)$ exists and defines a continuous function, see Corollary~\ref{cvp-19} and Lemma~\ref{cvp-13}. We set 
\begin{gather*}
	\cei{0}{}{f}g
	:= \sum_{i=0}^\infty \Kcal^i\left[\rli{0}{}{f}g\right].
\end{gather*}

\medskip\noindent
\emph{Step~2.}  We have the following estimate
\begin{equation}\label{cvp-e14}
	\Kcal^i \left[\bar{\mu}^{-1} g\right](x)
	\leq q^{i-1} \|g\|_{C[0,T]} \int_0^x k(r)\,dr,
\end{equation}
which we will prove by induction. If $i=1$, we get from the definition of the operator $\Kcal$ that 
\begin{align*}
	\Kcal\left[\bar{\mu}^{-1}g\right](x)
	= \int_0^x \bar{\mu}(r)\bar{\mu}^{-1}(r)g(r)k(x-r)\,dr
	\leq \|g\|_{C[0,T]} \int_0^x k(r)\,dr.
\end{align*}
We assume that \eqref{cvp-e14} holds for some $i\in \nat$. The induction step $i \rightsquigarrow i+1$ is achieved by
\begin{align*}
	\Kcal^{i+1}\left[\bar{\mu}^{-1}g\right](x)
	&=\Kcal^i \Kcal\bar{\mu}^{-1}g(x)\\
	&\leq \|g\|_{C[0,T]} \Kcal^i  \int_0^x k(r)\,dr\\
	&\leq \|g\|_{C[0,T]} q^i \int_0^T k(r)\,dr,
\end{align*}
where the second inequality follows from Corollary~\ref{cvp-19}. This completes the induction. 

\medskip\noindent
\emph{Step~3.} We show that $\cei{0}{}{f}g\in C_{\bar{\mu}}[0,T]$. We have already seen Step~1 that $\cei{0}{}{f}g\in C[0,T]$. In order to show $\cei{0}{}{f}g\in  C_{\bar{\mu}}(0,T]$, we use Lemma~\ref{cvp-21} and write $\cei{0}{}{f}g=\rli{0}{}{f}\hat g$ with the function $\hat g=\bar{\mu} \sum_{i=0}^{\infty}\Kcal^i\left[\bar{\mu}^{-1}g \right]$. In view of Theorem~\ref{bd-17} it is enough to show that $\hat g\in C(0, T]\cap L^1(0, T]$. The integrability of $\hat g$ follows from
\begin{align*}
	\int_0^T & \left|\sum_{i=0}^{\infty}\bar{\mu}(x) \Kcal^i\left[\bar{\mu}^{-1}g\right](x)\right| dx \\
	&\leq \int_0^T \bar{\mu}(x) \sum_{i=0}^{\infty} \left|\Kcal^i\left[\bar{\mu}^{-1}g\right](x)\right| dx \\
	&\leq\int_0^T g(x)\,dx + \frac {\|g\|_{C[0,T]}}{1-q}\int_0^T\int_0^x k(s)\,ds\; \bar{\mu}(x) \,dx \\
	&\leq T\|g\|_{C[0,T]} + \frac {\|g\|_{C[0,T]}}{1-q}\int_0^T\bar{\mu}(x) \,dx\int_0^T k(s)\,ds 
	<\infty, 
\end{align*}
where the second inequality follows from \eqref{cvp-e14} and the geometric series.

In order to see the continuity of $\hat g$, it suffices to show that $\sum_{i=0}^{\infty}\Kcal^i\left[\bar{\mu}^{-1}g \right]$ is continuous, since $\bar\mu$ is continuous on $C(0, T]$. Notice that $\bar\mu^{-1} g\in C[0,T]$, since the limit $\bar\mu(0+)\in (0,\infty]$ exists. We have
\begin{gather*}
	\sum_{i=0}^{\infty}\Kcal^i\left[\bar{\mu}^{-1}g \right]
	=\bar{\mu}^{-1}g +\sum_{i=1}^{\infty}\Kcal^i\left[\bar{\mu}^{-1}g \right],
\end{gather*}
where $\bar\mu^{-1}g\in C[0,T]$ and each $\Kcal^i\left[\bar{\mu}^{-1}g \right]\in C[0,T]$, see Lemma~\ref{cvp-15}. By Corollary~\ref{cvp-19} the series converges uniformly, and we conclude that $\sum_{i=0}^{\infty}\Kcal^i\left[\bar{\mu}^{-1}g \right]\in C[0,T]$

Thus, $\cei{0}{}{f}g\in C[0,T]\cap C_{\bar\mu}(0, T]= C_{\bar\mu}[0, T]$.

\medskip\noindent
\emph{Step~4.} As $\bar{\phi}$ is in $C_\mu[0,T]$, we can apply $\rld{0}{}{f}$ and obtain
\begin{align*}
\rld{0}{}{f}\bar{\phi}
&=\rld{0}{}{f}\rli{0}{}{f}\left[\bar{\mu} \sum_{i=0}^{\infty}\Kcal^i\left[\bar{\mu}^{-1}g \right]\right]\\
&=\bar{\mu} \sum_{i=0}^{\infty}\Kcal^i\left[\bar{\mu}^{-1}g \right]\\
&=g+\bar{\mu} \sum_{i=1}^{\infty}\Kcal^i\left[\bar{\mu}^{-1}g \right]\\
&=g+\bar{\mu} \sum_{i=1}^{\infty}\Kcal^{i-1}\left[\rli{0}{}{f}g \right]\\
&=g+\bar{\mu}\bar{\phi}
\end{align*}

\medskip\noindent
\emph{Step~5} Let us finally show that the solution to the initial value problem \eqref{cvp-e08} is unique in $C_{\bar{\mu}}[0,T] $. Let $\phi_1,\phi_2\in C_{\bar{\mu}}[0,T] $ be two solutions to \eqref{cvp-e08}. By the linearity of the operator $\ced{0}{}{f}$, $\psi:=\phi_1-\phi_2\in C_{\bar{\mu}}[0,T]$ satisfies the following equation
\begin{equation}\label{cvp-e16}
\left\{\begin{aligned}
		\rld{0}{}{f}\psi(x)& =\psi(x)\bar{\mu}(x), &&x>0;\\
		\psi(0)&=0.
	\end{aligned}\right.
\end{equation}
With the argument from Step~4, we can apply the inverse operator of $\rld{0}{}{f}$ on both sides, and get $\psi=\rli{0}{}{f}\left[\psi \bar{\mu}\right]=\Kcal\psi$. We show that $\psi= 0$ on $[0,T]$, hence $\phi_1=\phi_2$.  Assume, to the contrary, that  $\psi\neq 0$.  Because $\psi (x)=\Kcal\psi(x)$ we have 
\begin{gather*} 
		\int_0^x (\psi(x)-\psi(r))\bar{\mu}(r)k(x-r)\,dr=0,\quad\text{for all\ \ }  x\in [0,T].
\end{gather*} 
Take $ \xi\in \argmax_{r\in[0,T]}|\psi(r)|$. Then we have $ \int_0^\xi (\psi(\xi)-\psi(r))\bar{\mu}(r)k(\xi-r)\,dr=0$. This implies that $\psi(r)=\psi(\xi)$ for all $r\in [0,T]$. As $\psi(0)=0$, it follows that $\psi=0$, and the proof is completed.
\end{proof}

\begin{remark}\label{cvp-23} 
	The proof of Theorem \ref{cvp-09} shows that $\cei{0}{}{f}g\in C_{\bar{\mu}}[0, T]$ for any $g\in C[0, T]$.
\end{remark}

\section{Resolvent equation}\label{re}

In order to show that $-\ced{0}{}{f}$ is the generator of a stochastic process, we use the Hille--Yosida theorem. This means that we have to solve the resolvent equation \eqref{re-e02} below.
\begin{theorem}\label{re-03} 
	Let $(\bar\mu,k)$ be a positive Sonine pair and $\Lscr(\bar\mu; \lambda)=f(\lambda)/\lambda$,  $\Lscr(k; \lambda)=1/f(\lambda)$, where $f$ is a Bernstein function satisfying the Assumptions~\ref{A1} and~\ref{A2}. Moreover, assume that $\limsup_{x\to 0}\bar{\mu}(x)\int_0^x k(s)\,ds<1$. Then, for any $\phi_0\in \real$ and $\lambda\in\real\setminus\{0\}$, the following initial value problem  
	\begin{equation}\label{re-e02}
	\left\{
	\begin{aligned}
		\ced{0}{}{f}\phi(x) &=\lambda \phi(x), &&x\in (0, T],
	\\ \phi(0) &=\phi_0,
	\end{aligned}
	\right.
	\end{equation}
	has a unique solution $\phi$ in $C_{\bar{\mu}}[0,T]$, which is given by $\phi(x)=\phi_0\sum_{i=0}^{\infty}\left(\lambda \cei{0}{}{f}\right)^i \I(x)$.
\end{theorem}

\begin{proof}
Assume, for a moment, that $\phi\in C_{\bar\mu}[0,T]$ is a solution to \eqref{re-e02}. Applying Theorem~\ref{cvp-09} shows that
\begin{gather*}
	\phi(x) = \phi_0 + \cei{0}{}{f}\left[\ced{0}{}{f}\phi\right](x) = \phi_0 + \lambda\cei{0}{}{f}\phi(x).
\end{gather*}
Repeatedly inserting this equality into itself shows that any solution $\phi$ is necessarily of the form $\phi(x)=\phi_0\sum_{i=0}^{\infty} \left(\lambda \cei{0}{}{f}\right)^i \I(x)$.

We are now going to show that this series converges in $C_{\bar\mu}[0,T]$, establishing in this way the existence and uniqueness of the solution to \eqref{re-e02}. We use a Picard iteration scheme:
\begin{equation}\label{re-e04}
	\left\{
	\begin{aligned}
	\phi_{n}(x) &=\lambda \cei{0}{}{f}\phi_{n-1}(x)+\phi_0,\\
	\phi_{0}(x) &=\phi_0.
	\end{aligned}
	\right.
\end{equation}
Obviously,
\begin{gather*}
	\phi_{n+1}(x) 
	=\phi_0\sum_{i=0}^{n}\left(\lambda \cei{0}{}{f}\right)^i\I(x),
\end{gather*}
and we have to show that $\sum_{i=0}^{\infty}\lambda^i \left(\cei{0}{}{f}\right)^i \I$ converges and defines an element of $C_{\bar\mu}[0,T]$.

We use induction to show that for all $i\in\nat$
\begin{equation}\label{re-e06}
	\left(\cei{0}{}{f}\right)^i \I(x) 
	\leq \left(\frac{1}{1-q}\right)^i \left(\rli{0}{}{f}\right)^{i-1} K(x),
	\quad\text{where}\quad K(x) := \int_0^x k(y)\,dy.
\end{equation}
We have $q := \sup_{x\in [0,T]}\bar{\mu}(x)\int_0^x k(s)\,ds<1$, since $\limsup_{x\to 0}\bar{\mu}(x)\int_0^x k(s)\,ds<1$, see Corollary~\ref{cvp-19}. Recall from \eqref{cvp-e10} that $\cei{0}{}{f}\I(x) = \sum_{i=0}^\infty \Kcal^i \left[\rli{0}{}{f} \I\right](x)$. For $i=1$,  we get from  Corollary \ref{cvp-19}  that 
\begin{align*}
	\Kcal^i \left[\rli{0}{}{f} \I\right](x) 
	\leq q^i K(x),
\end{align*}
for all $i\in \nat_0$ and
\begin{gather*}
	\cei{0}{}{f}\I(x)
	= \sum_{i=0}^\infty \Kcal^i \left[\rli{0}{}{f} \I\right](x) 
	\leq \sum_{i=0}^\infty q^i K(x)
	= \frac{1}{1-q} K(x).
\end{gather*}
Using \eqref{re-e06} as induction assumption for some $i\in \nat$, we get  for $i \rightsquigarrow i+1$ that
\begin{align*}
	\left(\cei{0}{}{f}\right)^{i+1}\I(x)
	&=\cei{0}{}{f} \left(\cei{0}{}{f}\right)^{i}\I(x)\\
	&\leq \frac{1}{(1-q)^i} \cei{0}{}{f} \left(\rli{0}{}{f}\right)^{i-1} K(x)\\
	&=\frac{1}{(1-q)^i} \sum_{{n=0}}^\infty \Kcal^n \left\{\rli{0}{}{f} \left[\left(\rli{0}{}{f}\right)^{i-1} K\right]\right\}(x)\\
	&\leq\frac{1}{(1-q)^{i+1}} \left(\rli{0}{}{f}\right)^iK(x),
\end{align*}
where we use \eqref{cvp-e10} for the second equality; the last inequality follows from 
\begin{align*}
	\Kcal^n \rli{0}{}{f} \left[\left(\rli{0}{}{f}\right)^{n-1} K\right](x)
	&=\Kcal^{n-1} \Kcal \left[\left(\rli{0}{}{f}\right)^{i}K\right](x)\\	
	&=\Kcal^{n-1} \rli{0}{}{f}\left\{ \bar{\mu} \left[\left(\rli{0}{}{f}\right)^{i} K\right]\right\}(x)\\
	&\stackrel {(*)}{\leq} \Kcal^{n-1} \rli{0}{}{f}  \left[ \left(\rli{0}{}{f}\right)^{i} \left(\bar{\mu} K\right)\right](x)\\
	&\leq q \Kcal^{n-1} \left(\rli{0}{}{f}\right)^{i+1} \I(x)\\
	&= q \Kcal^{n-1} \left(\rli{0}{}{f}\right)^{i} K(x).
\end{align*}
In the step marked by $(*)$ we use the monotonicity of $\bar\mu$ and the integral representation of the operator $\rli{0}{}{f}$. Repeated  use of  the above calculation yields
\begin{gather*}
	\Kcal^n \rli{0}{}{f} \left[\left(\rli{0}{}{f}\right)^{i-1} K\right](x)
	\leq q^n \left(\rli{0}{}{f}\right)^{i} K(x).
\end{gather*}
This finishes the proof of \eqref{re-e06}. 

Now we show how the assertion of the theorem follows from \eqref{re-e06}.  Taking the  Laplace transform on the right hand side of \eqref{re-e06},  we obtain for $i\in\nat$
\begin{align*}
	\Lscr\left(\left(\rli{0}{}{f}\right)^{i-1} K; s\right)
	&= \left[\Lscr \left(k; s\right)\right]^{i-1} \Lscr \left(K; s\right)\\
	&=\frac1{f^{i-1}(s)}\frac1{sf(s)}
	=\frac1{sf^i(s)}.
\end{align*}
Define $F_n(x) :=\sum_{i=0}^n \left(\frac{|\lambda|}{1-q}\right)^i \left(\rli{0}{}{f}\right)^{i-1} K(x)$. We get 
\begin{align*}
	\Lscr\left(F_n; s\right) 
	&= \Lscr\left(\sum_{i=0}^n \left(\frac{|\lambda|}{1-q}\right)^i \left(\rli{0}{}{f}\right)^{i-1} K; s\right)\\
 	&= \sum_{i=0}^n \left(\frac{|\lambda|}{1-q}\right)^i \Lscr\left( \left(\rli{0}{}{f}\right)^{i-1} K; s\right)\\
 	&=\sum_{i=0}^n \left(\frac{|\lambda|}{1-q}\right)^i \frac1{sf^i(s)}.
\end{align*}
If  $s\geq s_0$ is large enough to guarantee that $|\lambda|/[(1-q)f(s)]<1$, we see that
\begin{gather*}
	\lim_{n\to\infty}\Lscr\left(F_n(x); s\right)
	= \sum_{i=0}^\infty \left(\frac{|\lambda|}{1-q}\right)^i \frac1{sf^i(s)}\quad\text{for all\ \ } s\geq s_0.
\end{gather*}
For all $m,n \in\nat$, $m>n$ we obtain
\begin{align*}
	\|F_m-F_n\|_{L^1(e^{-s_0 x}\,dx)}
	&= \int_0^\infty e^{-s_0 x}\left(F_m(x)-F_n(x)\right)dx\\
	&= \sum_{i=n+1}^m \int_0^\infty e^{-s_0 x}\left(\frac{|\lambda|}{1-q}\right)^i \left(\rli{0}{}{f}\right)^{i-1} K(x) dx\\
	&= \sum_{i=n+1}^m \left(\frac{|\lambda|}{1-q}\right)^i \frac1{s_0f^i(s_0)},
\end{align*}
which implies that $(F_n)_{n\in \nat}$ is a Cauchy sequence in $L^1(e^{-s_0x}\,dx)$. Thus, there exists a subsequence $(F_{n_k})_{k\in\nat}$, which converges almost everywhere to a function $F\in L^1(e^{-s_0x}\,dx)$. Without loss of generality we may assume that the subsequence converges at the end-point $x=T$ of the interval $[0,T]$, i.e.\ $\lim_{k\to\infty} F_{n_k}(T) = F(T)$. For $n < m$ we obtain that
\begin{align*}
	\sup_{x\in [0, T]}\left|F_m(x)-F_n(x)\right|
	&=\sup_{x\in [0, T]}\left[\sum_{i=n+1}^m \left(\frac{|\lambda|}{1-q}\right)^i  \left(\rli{0}{}{f}\right)^{i-1} K(x)\right] \\
	&\leq \sum_{n+1}^m\left(\frac{|\lambda|}{1-q}\right)^i  \left(\rli{0}{}{f}\right)^{i-1} K(T) \\
	&= F_m(T)-F_n(T).
\end{align*} 
In the last estimate we use the fact that the operator $\rli{0}{}{f}$ is monotonicity preserving for positive functions. With the same reasoning we get for $n_k\leq n+1<m\leq n_l$ that
\begin{gather*}
	\lim_{n,m\to\infty} \sup_{x\in [0, T]}\left|F_m(x)-F_n(x)\right|
	\leq \lim_{n,m\to\infty} \left[F_m(T)-F_n(T)\right]
	\leq \lim_{k,l\to\infty} \left[F_{n_l}(T)-F_{n_k}(T)\right] = 0.
\end{gather*}
Now we define 
\begin{gather*}
	G_n(x) :=\sum_{i=0}^n \left(\lambda \cei{0}{}{f}\right)^i \I(x).
\end{gather*}
Using \eqref{re-e06}, we get $G_n(x)\leq F_n(x)$ and furthermore, assuming $m>n$ we obtain  uniformly for all $x\in [0, T]$,
\begin{align*}
	\left|G_m(x)-G_n(x)\right|
	&= \left|\sum_{i=n+1}^m \left(\lambda \cei{0}{}{f}\right)^i \I(x)\right|\\
	&\leq \sum_{i=n+1}^m \left(\frac{|\lambda|}{1-q}\right)^i  \left(\rli{0}{}{f}\right)^{i-1} K(T) \\
	&= F_m(T)-F_n(T)
	\xrightarrow[n, m\to \infty]{}0.
\end{align*}
From this we conclude that  $G_n$ converges uniformly on $[0,T]$ to a function $G$.  Thus,  
\begin{gather*}
	G(x)=\phi_0\sum_{i=0}^{\infty}\left(\lambda \cei{0}{}{f}\right)^i \I(x) 
	\quad \text{and} \quad 
	G\in C[0, T].
\intertext{Finally,}
	G(x)
	= \phi_0 + \phi_0\sum_{i=1}^{\infty}\left(\lambda \cei{0}{}{f}\right)^i \I(x)
	= \phi_0 + \phi_0\lambda \cei{0}{}{f} \left[\sum_{i=0}^{\infty}\left(\lambda \cei{0}{}{f}\right)^i \I\right](x)
	= \phi_0 + \lambda \cei{0}{}{f} G(x),
\end{gather*}
which shows that $G$ solves \eqref{re-e02} and, by Theorem \ref{cvp-09} and Remark~\ref{cvp-23}, that $G\in C_{\bar\mu}[0, T]$. This finishes the proof.
\end{proof}

\begin{remark}
As the Picard-Iteration is equivalent to Banach's fixed point theorem, one can also prove Theorem \ref{re-03} by an application of the Banach fixed-point theorem. Let us sketch the argument: Fix $T>0$ and define for $u\in C[0,T]$ the operator.
\begin{align*}
	Su(x):=\phi_0+\lambda\cei{0}{}{f}u(x),\,x>0.
\end{align*}
In this case we can rewrite \eqref{re-e02} in the following fixed-point form
\begin{align*}
	\phi(x) = S \phi(x),\;x\in [0,T].
\end{align*} 
We have already seen that the operator $S$ satisfies 
\begin{align*}
	|Su(x)-Sv(x)|
	\leq  |\lambda| \cei{0}{}{f}|u-v|(x)
	\leq \|u-v\|_{C[0,T]} \frac{|\lambda|}{1-q}K(x).
\end{align*}
As $K(x)\to 0$ as $x\to 0$, there exists an $\epsilon>0$ such that 
\begin{align*}
	\sup_{x\in [0,\epsilon]}K(x)\leq \frac{1-q}{2|\lambda|},
\end{align*}
which implies that there exists a unique $\phi_1\in C[0,\epsilon]$ such that
\begin{align*}
	\phi_1(x)= S \phi_1(x),\;x\in [0,\epsilon].
\end{align*}
Fix $n\in \nat$ and assume there exists a $\phi_{n}\in C[0,n\epsilon]$ such that
\begin{align*}
\phi_{n}(x)=S\phi_n(x)\quad\text{for all\ \ } x\in [0,n\epsilon].
\end{align*}
We define
\begin{align*}
	C_{00}[n\epsilon,(n+1)\epsilon]
	:=\left\{u\in C[0,(n+1)\epsilon]\::\: u(x)=\phi_n(x)\text{\ \ for all\ \ }x\in [0,n\epsilon]\right\}
\end{align*}
and $S_{n+1}u(x):=Su(x)$ for $x\in [0,(n+1)\epsilon]$. Observe that $S_{n+1}u(x)=S\phi_n(x)=\phi_n(x)=u(x)$ for all $x\in [0,n\epsilon]$ and $u\in C_{00}[n\epsilon,(n+1)\epsilon]$. Thus, for for $u,v\in C_{00}[n\epsilon,(n+1)\epsilon]$ we have that
\begin{align*}
	\rli{0}{}{f}(u-v)(x)=\int_0^xk(x-r)(u(r)-v(r))\,dr=0
\end{align*}
for all $x\leq n\epsilon$ and
\begin{align*}
	\left|\rli{0}{}{f}(u-v)(x)\right|
	\leq \int_{n\epsilon}^x k(x-r) |u(r)-v(r)|\,dr
	\leq\int_{n\epsilon}^x k(x-r)\,dr\,\|u-v\|_{C[0,(n+1)\epsilon]}. 
\end{align*}
We obtain that for all $x\in [n\epsilon,(n+1)\epsilon]$
\begin{align*}
	\left|\mathcal{K}\left(\rli{0}{}{f}u\right)(x) - \mathcal{K}\left(\rli{0}{}{f}v\right)(x)\right|
	&\leq \int_{n\epsilon}^x k(x-r)\bar{\mu}(r) \left|\rli{0}{}{f}(u-v)(r)\right| dr\\
	&\leq \int_{n\epsilon}^x k(x-r)\bar{\mu}(r) \int_{n\epsilon}^r k(r-s)\,ds\,dr\,\|u-v\|_{C[0,(n+1)\epsilon]}\\
	&\leq \int_{n\epsilon}^x k(x-r)\bar{\mu}(r) \int_{0}^rk(r-s)\,ds\,dr\,\|u-v\|_{C[0,(n+1)\epsilon]}\\
	&\leq q\int_{n\epsilon}^x k(x-r)\,dr\|u-v\|_{C[0,(n+1)\epsilon]}.
\end{align*}
By iterating this argument we obtain for $i\in \nat$ that
\begin{align*}
	\left|\mathcal{K}^i\left(\rli{0}{}{f}u\right)(x) - \mathcal{K}^i\left(\rli{0}{}{f}v\right)(x)\right|
	&\leq q^i\int_{n\epsilon}^x k(x-r)\,dr\,\|u-v\|_{C[0,(n+1)\epsilon]}\\
	&\leq q^i \int_{0}^{x-n\epsilon}k(r)\,dr\,\|u-v\|_{C[0,(n+1)\epsilon]}\\
	&\leq q^i K(\epsilon)\|u-v\|_{C[0,(n+1)\epsilon]}.
\end{align*}
Summing up we obtain that
\begin{align*}
	\left|\cei{0}{}{f}u(x)-\cei{0}{}{f}v(x)\right|
	\leq K(\epsilon)\sum_{i=0}^\infty q^i \|u-v\|{\infty}=\frac{K(\epsilon)}{1-q}\|u-v\|_{C[0,(n+1)\epsilon]}
\end{align*}
for all $x\in [n\epsilon,(n+1)\epsilon]$. We conclude that there exists a unique $\phi_{n+1}\in C_{00}[n\epsilon,(n+1)\epsilon]$ such that
\begin{align*}
	\phi_{n+1}(x)=S\phi_{n+1}(x)
\end{align*}
for all $x\in [0,(n+1)\epsilon]$. By taking $n$ large enough, we find a unique solution of \eqref{re-e02} in $[0,T]$. This proof also shows that the solution can be extended to $[0,\infty)$ and also to nonlinear equations of the type
\begin{align*}
\left\{
	\begin{aligned}
		\ced{0}{}{f}\phi(x) &=g(\phi(x))+h(x), &&x\in (0, T],
	\\ \phi(0) &=\phi_0,
	\end{aligned}
	\right.
\end{align*}
have a unique solution if $g,h\in C[0,T]$, such that $g$ globally Lipschitz continuous in $[0,T]$.
\end{remark}

For the inhomogeneous resolvent equation we have the following result. 
\begin{theorem}\label{re-05} 
	Let $(\bar\mu,k)$ be a Sonine pair and $\Lscr(\bar\mu; \lambda)=f(\lambda)/\lambda$, $\Lscr(k; \lambda)=1/f(\lambda)$, where $f$ is a Bernstein function satisfying the Assumptions~\ref{A1} and~\ref{A2}. Moreover, assume that $\limsup_{x\to 0}\bar{\mu}(x)\int_0^x k(s)\,ds<1$. Then, for any $\phi_0\in \real$, $\lambda\in\real\setminus\{0\}$ and $g\in C[0,T]$, the following initial value problem  
	\begin{equation}\label{re-e08}
		\left\{
		\begin{aligned}
			\ced{0}{}{f}\phi(x) &=\lambda \phi(x)+g(x), &&x\in (0, T],\\
			\phi(0)&=\phi_0, 
		\end{aligned}
		\right.
	\end{equation}
	has a unique solution in $C_{\bar{\mu}}[0,T]$, which is given by 
\begin{gather*}
	\phi(x)
	= \phi_0 \sum_{i=0}^{\infty}\lambda^i \left(\cei{0}{}{f}\right)^i \I(x)
	+ \sum_{i=0}^{\infty}\lambda^i \left(\cei{0}{}{f}\right)^{i+1}g(x).
\end{gather*}
\end{theorem}
\begin{proof} 
	The uniqueness follows from the linearity of the operator $\ced{0}{}{f}$ and the uniqueness of the homogeneous problem \eqref{re-e02}.
	
	Using the proof of Theorem \ref{re-03}, we can show that the series
	\begin{gather*}
		\sum_{i=0}^{\infty}\lambda^i\left(\cei{0}{}{f}\right)^{i} \I
		\quad \text{and}\quad
		\sum_{i=0}^{\infty}\lambda^i \left(\cei{0}{}{f}\right)^{i+1} g 
	\end{gather*} 
	converge (absolutely and) uniformly in $[0,T]$. Therefore, $\phi\in C[0,T]$. In particular,
\begin{align*}
	\lambda \cei{0}{}{f}\phi + \cei{0}{}{f}g
	&\pomu{(*)}{=}{}\lambda \phi_0 \cei{0}{}{f}\left[\sum_{i=0}^{\infty} \lambda^i \left(\cei{0}{}{f}\right)^i \I\right] 
	+ \lambda \cei{0}{}{f} \left[\sum_{i=0}^{\infty}\lambda^i \left(\cei{0}{}{f}\right)^{i+1} g\right] + \cei{0}{}{f}g\\
	&\omu{(*)}{=}{}  \phi_0 \sum_{i=0}^{\infty} \lambda^{i+1} \left(\cei{0}{}{f}\right)^{i+1}\I 
	+ \sum_{i=0}^{\infty}\lambda^{i+1} \left(\cei{0}{}{f}\right)^{i+2} g + \cei{0}{}{f}g\\
	&\pomu{(*)}{=}{}  \phi_0\sum_{i=1}^{\infty}\lambda^i \left(\cei{0}{}{f}\right)^{i}\I 
	+ \sum_{i=0}^{\infty}\lambda^{i} \left(\cei{0}{}{f}\right)^{i+1} g\\
	&\pomu{(*)}{=}{}\phi-\phi_0.
\end{align*}
In the step marked by (*) we use the continuity of the operator $\cei{0}{}{f}$ which is clear from its construction using the Picard scheme, i.e.\ the Banach fixed--point theorem. Theorem \ref{cvp-09} now shows that $\phi$  solves \eqref{re-e08}.
\end{proof}

\section{Construction of the censored process}\label{sp}

We will now give a probabilistic construction of the \textbf{censored process} $S^c = (S_t^c)_{t\geq 0}$. We will see in Section~\ref{pr} that $-\ced{0}{}{f}$ is the infinitesimal generator of $S^c$. Throughout this section, $S = (S_t)_{t\geq 0}$ is the subordinator (increasing L\'evy process) whose Laplace exponent is the Bernstein function $f$. We assume that $f$ satisfies the Assumptions~\ref{A1} and~\ref{A2}. The censored process will be constructed by the piecing-out procedure due to Ikeda, Watanabe \& Nagasawa~\cite{1966_Nobuyuki}. Roughly speaking, the paths $t\mapsto S_t^c$ are obtained from  $t\mapsto x-S_t$, $x>0$, by deleting (i.e.\ ``censoring'') all jumps that make the path negative.

To make this rigorous, let $S^n = (S_t^n)_{t\geq 0}$ be independent copies of $S$ and fix a starting point $x>0$. This also fixes the probability space $(\Omega,\mathscr{A},\Pp)$, which is large enough to accommodate $S$ along with the i.i.d.\ copies $S^1, S^2,\dots$. Note that under $\Pp$ we have $S_0 = 0 = S^n_0$ for all $n\in\nat$. Pick a starting point $S_0^c = x\in (0,T]$. From $x$ we run $x-S^1_t$ until $\sigma_1$, the first exit time of the subordinator $S^1$ from $[0,x]$. Now we use $x-S^1_{\sigma_1-}\in (0,x)$ as the starting point of $-S^2$, and run the process $(x-S^1_{\sigma_1-}) - S^2_t$ until it exits $[0,x]$, i.e.\ until the stopping time $\sigma_2$, which is the first time such that $S^2_t > x-S^1_{\sigma_1-}$ etc. The stopping times $\tau_n := \sigma_1 + \dots + \sigma_n$, $n\in\nat$, are called the \textbf{censoring times}, and we set $\tau_0:=0$.

More formally, the censored decreasing subordinator $S^c$ is given by 
\begin{gather}\label{sp-e02}
S_t^c =
\left\{\begin{aligned}
    &x-S_t^1, && 0\leq t<\tau_1,  &&n=1, \\
    &S^c_{\tau_{n-1}-}-S^n_{t-\tau_{n-1}}, &&\tau_{n-1} \leq t<\tau_n, && n\geq 2, \\
    &\partial, &&t\geq \smash[b]{\tau_\infty := \sup_{n\in\nat}\tau_n,}
\end{aligned}\right.
\end{gather}
and the waiting times between two censoring times, $\sigma_n = \tau_n-\tau_{n-1}$, $n\in\nat$, are
\begin{gather}\label{sp-e04}\begin{aligned}
	\sigma_n 
	&= \inf\left\{t>0\::\: S_{t+\tau_{n-1}}^c \leq 0\right\}\\
	&= \inf\left\{t\geq 0\::\: S_t^n > S^c_{\tau_{n-1}}\right\}
	= E_n\left(S^c_{\tau_{n-1}}\right),
\end{aligned}\end{gather}
where $E_n$ is the generalized inverse of $t\mapsto y-S^n_t$, i.e.\  $E_n(y)=\inf\left\{t\geq 0\::\:S^n_t>y\right\}$, $y>0$. In particular, $\sigma_n$ and $\tau_n$ are stopping times. Observe that the censored process is continuous at the censoring times, i.e.\ $S^c_{\tau_n-} = S^c_{\tau_n}$, $n\in\nat$.

\begin{remark}\label{sp-02}
	Notice that we \textbf{kill} the process $S_t^c$ at $t=\tau_\infty$, i.e.\ we pick a cemetery point $\partial$ (usually from the one-point compactification of the state space $(0,T]$), define $S_t^c = \partial$ for all $t\geq\tau_\infty$ and, if needed, we extend all functions $\phi$ by setting $\phi(\partial):=0$.
\end{remark}

From \eqref{sp-e02} and \eqref{sp-e04} we see that $S^c_\bullet$ can be represented as $S^c_\bullet = \Psi(x-S^1_\bullet,-S^2_\bullet,-S^3_\bullet,\dots)$ where $\Psi$ is a suitable functional. In particular, we can define probability measures $\left(\Pp^x\right)_{x>0}$ via $\Pp^x\left(S^c_\bullet\in\Gamma\right) := \Pp\left(\Psi(x-S^1_\bullet,-S^2_\bullet,-S^3_\bullet,\dots)\in \Gamma\right)$ ($\Gamma$ is a cylinder set in $\real^{[0,\infty)}$) where $x>0$ is the starting point of $S^c$, i.e.\ $\Pp^x\left(S_0^c = x\right)=1$. We will switch between $\Pp^x$ and $\Pp$ as appropriate.

\begin{theorem}\label{sp-03}
	The censored process $S^c$ is a strong Markov process.
\end{theorem}
\begin{proof}
Consider the natural filtration $(\Fcal_t)_{t\geq 0}$ of $S^c$ and let $\eta$ be a stopping time. Pick $T>0$ and $x\in (0,T]$. We have to show that for all bounded measurable functions $\phi:(0,T]\to\real$
\begin{equation}\label{sp-e06}
	\Ee^x\left[\phi(S^c_{\eta+s}) \mid \Fcal_\eta\right]
	= \Ee^{S^c_\eta}\left[\phi(S^c_{s})\right], \quad s>0.
\end{equation}
Set $A_n := \left\{\tau_{n-1}\leq \eta <\tau_n\right\}$ and $B_i := \left\{\tau_{i-1}\leq\eta+s<\tau_i\right\}$. Clearly, $\Omega = \bigcup_{n=1}^\infty\bigcup_{i=n}^\infty A_n\cap B_i$ is a partitioning of $\Omega$ with mutually disjoint sets. Therefore, it is enough to consider \eqref{sp-e06} on $A_n\cap B_i$.

\medskip\noindent
\emph{Case~1: $n=i$.} We have 
\begin{align*}
	&\Ee^x\left[\phi(S^c_{\eta+s}) \I_{A_n}\I_{B_i} \mid \Fcal_\eta\right]\\
	&\quad\pomu{(*)}{=}{} \Ee^x\left[\phi\left(S^c_{\tau_{n-1}}-S^n_{\eta-\tau_{n-1}} + \left(S^n_{\eta-\tau_{n-1}}-S^n_{\eta+s-\tau_{n-1}}\right)\right) \I_{B_i} \I_{A_n} \mid \Fcal_\eta\right]\\
	&\quad\omu{(*)}{=}{} \Ee^x\left[\phi\left(S^c_{\eta}+ \left(S^n_{\eta-\tau_{n-1}}-S^n_{\eta+s-\tau_{n-1}}\right)\right) \I_{B_i}\I_{A_n} \mid \Fcal_\eta\right]\\
	&\quad\pomu{(*)}{=}{} \Ee^x\left[\phi\left(S^c_{\eta} - \left(S^n_{\eta+s-\tau_{n-1}}-S^n_{\eta-\tau_{n-1}}\right)\right) \I_{B_i} \mid \Fcal_\eta\right] \I_{A_n}\\
	&\quad\omu{(*)}{=}{} \Ee^{S^c_\eta}\left[ \phi\left(S_s^c\right) \I_{B_i}\right]\I_{A_n}.
\end{align*}
In the steps  marked with (*) we use the following facts:
\begin{itemize}\itemsep=3pt
	\item 
		$S^c_{\tau_{n-1}-}-S^n_{\eta-\tau_{n-1}} = S^c_\eta$ on $A_n$;
	\item 
		$\Ee^x\left[g(X, Y)|\Fcal_\eta\right]=\Ee^x\left[g(z, Y)\right]|_{z=X}$ if $g$ is bounded and measurable, $X$ is $\Fcal_\eta$-measurable and $Y$ is independent of $\Fcal_\eta$;
	\item 
		$S^n_{\eta-\tau_{n-1}}-S^n_{\eta+s-\tau_{n-1}}\sim -S^n_s\sim -S^1_s$ and $x-S^1_s = S^c_s$ $\Pp^x$-a.s.\ on $A_n\cap B_i$.
\end{itemize}

\smallskip\noindent
\emph{Case~2: $n<i$.} Using a telescoping argument we see that on $A_n\cap B_i$
\begin{align*}
	S^c_{\eta+s} 
	&= S^c_\eta + \left(S^c_{\tau_n}-S_\eta\right) + \left(S^c_{\tau_{n+1}}-S^c_{\tau_n}\right)
	+ \dots + \left(S^c_{\tau_{i-1}}-S^c_{\tau_{i-2}}\right) + \left(\smash{S^c_{\eta+s}}-S^c_{\tau_{i-1}}\right)\\
	&= S^c_\eta - \left(S^n_{\sigma_n-}-S^n_{\eta-\tau_{n-1}}\right) - S^n_{\sigma_{n+1}-} - \dots - S^{i-1}_{\sigma_{i-1}-} - S^i_{\eta+s-\tau_{i-1}}\\
	&\sim S^c_\eta - \left(\smash{\hat S^n_{(\tau_n-\eta)-}} + S^{n+1}_{\sigma_{n+1}-} + \dots + S^{i-1}_{\sigma_{i-1}-} + \smash{S^i_{\eta+s-\tau_{i-1}}}\right),
\end{align*}
where $\hat S^n$ is an independent copy of $S^n$. The expression in brackets is independent of $S^c_\eta$ and, on $A_i\cap B_n$, it has the same law as $S_s^c$ because of the construction of the censored process. Therefore,
\begin{align*}
	&\Ee^x\left[\phi(S^c_{\eta+s}) \I_{A_n}\I_{B_i} \mid \Fcal_\eta\right] \\
	&=\Ee^x\left[\phi\left(S^c_\eta - \big(\hat S^n_{(\tau_n-\eta)-} + S^{n+1}_{\sigma_{n+1}-} + \dots + S^{i-1}_{\sigma_{i-1}-} + S^i_{\eta+s-\tau_{i-1}}\big)\right) \I_{A_n}\I_{B_i} \mid \Fcal_\eta\right]\\
	&=\Ee^x\left[\phi\left(S^c_\eta - \big(\hat S^n_{(\tau_n-\eta)-} + S^{n+1}_{\sigma_{n+1}-} + \dots + S^{i-1}_{\sigma_{i-1}-} + S^i_{\eta+s-\tau_{i-1}}\big)\right) \I_{B_i} \mid \Fcal_\eta\right] \I_{A_n} \\
	&=\Ee^{S^c_\eta}\left[\phi\left(S^c_s\right) \I_{B_i}\right]\I_{A_n}.
\end{align*}
In the last step we argue as in Case~1.
\end{proof}

\begin{remark}\label{sp-05} 
	Using methods from probabilistic potential theory, see \cite[Thm.~14.8]{1988_Sharpe}, one can give another construction of the censored process. This construction requires the transfer kernel $K(S_{\tau_n}^c, dy) := \delta_{S_{\tau_n-}^c}(dy)$, where $\delta$ is the Dirac measure.
\end{remark}

Recall that ${f(\lambda)}^{-1}$ is the potential function of the subordinator $S=(S_t)_{t\geq 0}$. Since ${f(\lambda)}^{-1} = \Lscr\left[k;\lambda\right]$, the potential measure of $S$ is given by $U(dx) = k(x)\,dx$. The next lemma connects the potential measure and the kernels $k_1(x, r)=\bar\mu(r)k(x-r)$ and $k_i(x,r)$, $i\geq 2$, from Definition \ref{cvp-07} with (properties of) the censored process.
\begin{lemma}\label{sp-07}
	Let $(\bar\mu,k)$ be a positive Sonine pair and $\Lscr(\bar\mu; \lambda)=f(\lambda)/\lambda$, $\Lscr(k; \lambda)=1/f(\lambda)$, where $f$ is a Bernstein function satisfying the Assumptions~\ref{A1} and~\ref{A2}.  Moreover, assume that $\limsup_{x\to 0}\bar{\mu}(x)\int_0^x k(s)\,ds<1$.
	Fix $x>0$ and assume that $S_0^c=x$. For every $n\in\nat$ one has 
	\begin{enumerate}[label=\upshape\alph*),ref=\upshape\alph*),leftmargin=*,itemsep=4pt,itemindent=\parindent]
	\item\label{sp-07-a}
		$\Ee^x[\tau_n]<\infty$, $\Pp^x\left[S_{\tau_n}^c\in (0,x)\right]=1$ and $S_{\tau_n}^c$ has the probability density $k_n(x,\cdot)$;
	\item\label{sp-07-b}  
		$\Ee^x [\sigma_{n+1}] = \Ee^x\left[E_{n+1}(S_{\tau_n}^c)\right] = \int_0^x k_n(x,y)\,U(dy)$;
	\item\label{sp-07-c}
		$\tau_\infty:=\sup_{n\in\nat}\tau_n$ satisfies $\Ee^x\left[\tau_\infty\right]<\infty$, $\Pp^x\left[\tau_\infty<\infty\right]=1$ and $\Pp^x\left[S^c_{\tau_\infty-}=0\right] = 1$.
	\end{enumerate}
\end{lemma}
\begin{proof}
\ref{sp-07-a} We use induction. If $n=1$, then 
\begin{gather*}
	\Ee^x[\tau_1]
	= \Ee[E_1(x)]
	= U(0,x)
	= \int_0^x k(y)\,dy
	<\infty,
\end{gather*} 
where we use that $\tau_1 = E_1(x) = \inf\left\{s\geq 0\::\: x-S^1_s<0\right\}$ and $\Ee\left[E_1(x)\right]=U(0,x)$, see \cite[Ch.~I.4]{1996_Bertoin}. From  \cite[Prop.~III.2]{1996_Bertoin} we know that  
\begin{gather*}
	\Pp\left(S_{\tau(x)-}\in dy,\; S_{\tau(x)}\in dz\right)
	= U(dy)\,\mu(dz-y),
\end{gather*} 
where $\tau(x) = \inf\left\{t\geq 0\::\:S_t\geq x\right\}$. Since $\tau_1=E_1(x)=\tau(x)$ under $\Pp^x$, we get for $0\leq a\leq x$,
\begin{align*}
	\Pp^x\left(S^c_{\tau_1}\in (a, x]\right)
	&=\Pp\left(x-S_{\tau(x)-}\in (a, x]\right)\\
	&=\Pp\left(x-S_{\tau(x)-}\in (a, x], S_{\tau(x)}\geq x\right)\\
	&=\Pp\left(S_{\tau(x)-}\in [0, x-a), S_{\tau(x)}\geq x\right)\\
	&=\int_0^{x-a}\bar\mu(x-y)k(y)\,dy.
\end{align*}
This shows that $k_1(x,r) = \bar\mu(x-y)k(y)$ is the density of $S^c_{\tau_1-}$ under $\Pp^x$. In particular, $S^c_{\tau_1}\in (0,x)$ holds $\Pp^x$-almost surely.

Assume that the assertions stated in \ref{sp-07-a} hold for some $n\geq 1$. By construction,
\begin{gather*}
	\sigma_{n+1}
	=\inf \left\{r\geq 0\::\: S^c_{\tau_n}<S^{n+1}_{r}\right\}
	=E_{n+1}\left(S^c_{\tau_n}\right),
\end{gather*}
and $S^c_{\tau_n}$ is independent of $S^{n+1}$. Using $S^c_{\tau_n}<x$ and $E_{n+1}(x)\sim E_1(x)$, we see
\begin{gather*}
	\Ee^x\left[\tau_{n+1}\right]
	= \Ee^x\left[E_{n+1}(S^c_{\tau_n})\right] + \Ee^x\left[\tau_n\right]
	\leq \Ee\left[E_{n+1}(x)\right] + \Ee^x\left[\tau_n\right]
	= U(0,x) + \Ee^x\left[\tau_n\right]
	<\infty.
\end{gather*}  
Next we show that $S_{\tau_{n+1} }^c\in (0,x)$. Using the definition of $\sigma_{n+1}$, we have 
\begin{gather*}
	S_{\tau_{n+1}}^c
	= S_{\tau_n}^c - S^{n+1}_{\sigma_{n+1}-}
	= S_{\tau_n}^c -S^{n+1}_{E_{n+1}(S^c_{\tau_n})-} \in (0, S^c_{\tau_n}) \subset (0,x).
\end{gather*}
Finally, we show that $S_{\tau_{n+1} }^c$ has the probability density $k_{n+1}(x,\cdot)$. For any bounded measurable $\phi$, we have 
\begin{align*}
	\Ee^x \left[\phi\left(S_{\tau_{n+1}}^c\right)\right]
	&= \Ee^x \left[\phi\left(S_{\tau_n }^c - S^{n+1}_{E_{n+1}(S^c_{\tau_n})-}\right)\right]\\
	&= \int_0^x \Ee \left[\phi\left(y-S^{n+1}_{E_{n+1}(y)-}\right)\right] \Pp^x(S_{\tau_n }^c\in dy)\\
	&= \int_0^x \Ee\left[\phi\left(y-S^{n+1}_{E_{n+1}(y)-}\right)\right] k_n(x,y)\,dy\\
	&= \int_0^x \Ee^y\left[\phi\left(S^{c}_{E_{1}(y)}\right)\right]k_n(x,y)\,dy\\
	&= \int_0^x\int_0^y \phi(z)k_1(y,z)\,dz \: k_n(x,y)\,dy\\
	&= \int_0^x \phi(z)\int_z^x k_n(x,y)k_1(y,z)\,dy\,dz.
\end{align*}
In this calculation we use that the subordinators $S^n$ are i.i.d., $S^c_{E_1}(y)$ has under $\Pp^y$ the density $k_1(y,\cdot)$, and $S^c_{\tau_n}$ is independent of $S^{n+1}$. Finally, it follows from Definition \ref{cvp-07} that $S_{\tau_n}^c$ has the density $k_{n+1}(x,\cdot)$.

\medskip\noindent
\ref{sp-07-b} Using the results of Part~\ref{sp-07-a}, we have $\Ee^x\left[\sigma_{n+1}\right] = \Ee^x\left[E_{n+1}(S^c_{\tau_n})\right]$. Since $S^c_{\tau_n}$ and $S^{n+1}$ are independent, we have
\begin{align*}
	\Ee^x\left[\sigma_{n+1}\right]
	&=\Ee^x\left[E_{n+1}(S^c_{\tau_n })\right]\\
	&=\Ee\left[\int_0^xE_{n+1}(y)k_n(x,y)\,dy\right]\\
	&=\int_0^x\Ee\left[E_{n+1}(y)\right]k_n(x,y)\,dy\\
	&=\int_0^x U(0,y)k_n(x,y)\,dy.
\end{align*}

\medskip\noindent
\ref{sp-07-c} We show that $\Ee^x\left[\tau_\infty\right]<\infty$, which implies $\Pp^x\left(\tau_\infty<\infty\right)=1$. By monotone convergence,
\begin{gather*}
	\Ee^x \left[\tau_\infty\right]
	= \sum_{i=1}^\infty \Ee^x\left[\sigma_i\right]
	= \sum_{i=1}^\infty \int_0^x U(0,y) k_{i-1}(x,y)\,dy
	=\sum_{i=1}^\infty \Kcal^{i-1} U(0,x)<\infty.
\end{gather*}
For the last inequality, we use the definition of $U(0,x)=\int_0^x k(y)\,dy$ and Corollary~\ref{cvp-19}.

Next we show $\Pp^x\left(S^c_{\tau_\infty-}>0\right)=0$. We have
\begin{gather*}
	\Pp^x\left(S^c_{\tau_\infty-}>0\right)
	\leq \sum_{n=1}^\infty \Pp\left(S^c_{\tau_\infty-}>\tfrac1{n}\right). 
\end{gather*}
Since $\tau_i\leq\tau_\infty$, for every $i\in\nat$ the inclusion $\left\{S^c_{\tau_{\infty}-}>\frac1{n}\right\}\subset \left\{S^c_{\tau_i}>\frac1{n}\right\}$ holds; hence 
\begin{gather*}
	\Pp^x\left(S^c_{\tau_\infty-}>\tfrac1{n}\right)
	\leq \Pp^x\left(\bigcap_{i=1}^\infty\left\{S^c_{\tau_i}>\tfrac1{n}\right\}\right)
	= \lim_{i\to\infty}\Pp^x\left(S^c_{\tau_i}>\tfrac1{n}\right).
\end{gather*}
Since $K(y):=\int_0^y k(s)\,ds$ is increasing and satisfies $K(0)=0$, the Markov inequality shows that 
\begin{align*}
	\Pp^x\left(S^c_{\tau_i}>\tfrac1{n}\right)
	=\Pp^x\left(K(S^c_{\tau_i})\geq K\left(\tfrac1{n}\right)\right)
	&\leq \frac 1{K\left(\tfrac1{n}\right)} \Ee^x\left[K(S^c_{\tau_i})\right]\\
	&= \frac 1{K\left(\tfrac1{n}\right)} \int_0^x k_i(x,y) K(y)\, dy 
	\xrightarrow[i\to \infty]{}0.
\end{align*} 
Indeed, since $\sum_{i=1}^\infty\int_0^x k_i(x,y)K(y)\, dy = \sum_{i=1}^\infty\int_0^x k_i(x,y)\int_0^y k(s)\,ds\, dy < \infty$, cf.~Lemma~\ref{cvp-19}, we see that $\int_0^x k_i(x,y)K(y)\, dy\to 0$ as $i\to\infty$.
\end{proof}

\section{Probabilistic representation}\label{pr}

In this section we identify the generator of the censored process $S^c$. We continue to use the notation introduced in the previous section; in particular, $S^n$ are i.i.d.\ copies of the subordinator $S$ with the Bernstein function $f$, which satisfies the Assumptions~\ref{A1} and \ref{A2}. Moreover, we assume that $\limsup_{x\to 0}\bar\mu(x)\int_0^x k(y)\,dy<1$. By $\tau_n$ we denote the censoring times, $(\bar\mu,k)$ is the positive Sonine pair relating to $f$ and $U(dx) = k(x)\,dx$ is the potential measure of $S$.

\begin{lemma}\label{pr-03}  
	Let $S=(S_t)_{t\geq 0}$ be the subordinator with Bernstein function $f$ satisfying~\ref{A1} and~\ref{A2} and $T>0$. Then one has for all measurable functions $g:(0,T]\to [0,\infty)$ and $x\in (0,T]$
	\begin{equation}\label{pr-e02}
		\rli{0}{}{f} g(x)
		= \Ee\left[\int_0^{\tau_1} g(x-S_t)\,dt\right]
	\end{equation}
\end{lemma}
\begin{proof} 
Denote by $p_t(dy)$ the transition probability of $S_t$. Since $U(dy) = \int_0^\infty p_t(dy)\,dt$ (in the sense of vague convergence) and $x-S_t\sim p_t(x-dy)$, we have
\begin{align*}
	\Ee\left[\int_0^{\tau_1} g(x-S_t)\,dt\right]
	&= \Ee\left[\int_0^\infty g(-S_t) \I_{\{t\leq \tau_1\}}\,dt\right]\\
	&= \Ee\left[\int_0^\infty g(-S_t) \I_{\{x-S_t\leq 0\}}\,dt\right]\\
	&= \int_0^\infty\int_0^x  g(y) \Pp(x-S_t \in dy) \,dt\\
	&= \int_0^x \int_0^\infty g(y)\, p_t(x-dy) \,dt\\
	&= \int_0^x g(y)\, U(x-dy). 
\end{align*}
Since $U(dx)=k(x)\,dx$, the definition of the Bernstein--Riemann--Liouville integral shows 
\begin{gather*}
	\rli{0}{}{f} g(x) = \int_0^x g(y)k(x-y)\,dy = \Ee \left[ \int_0^{\tau_1} g(x-S_t) \, dt \right]. \qedhere
\end{gather*}
\end{proof}

\begin{theorem}\label{pr-05} 
	Let $(\bar\mu,k)$ be a positive Sonine pair and $\Lscr(\bar\mu; \lambda)=f(\lambda)/\lambda$, $\Lscr(k; \lambda)=1/f(\lambda)$, where $f$ is a Bernstein function satisfying the Assumptions~\ref{A1} and~\ref{A2}.  Moreover, assume that $\limsup_{x\to 0}\bar{\mu}(x)\int_0^x k(s)\,ds<1$.
	For $g\in C[0,T]$, and $x\in (0,T]$ the following representation for the censored fractional integral holds true
	\begin{equation}\label{pr-e04}
		\cei{0}{}{f}g(x)
		= \sum_{n=0}^\infty \Ee^x \left[\rli{0}{}{f}g(S^c_{\tau_n})\right]
		= \Ee^x \left[\int_0^{\tau_\infty}g(S^c_t)\,dt\right],
	\end{equation}
	where $\tau_\infty = \sup_{n\in\nat}$ and $\tau_1, \tau_2, \dots$ are the censoring times.
\end{theorem}
\begin{proof} 
	Assume first that $g\geq 0$. Then we have
\begin{align*}
	\Ee^x\left[\int_0^{\tau_\infty}g\left(S^c_t\right)dt\right]
	&\pomu{(*)}{=}{} \sum_{n=0}^\infty \Ee^x\left[\int_{\tau_n}^{\tau_{n+1}}g\left(S^c_t\right)dt\right]\\
	&\pomu{(*)}{=}{} \sum_{n=0}^\infty \Ee^x \left[\int_0^{\tau_{n+1}-\tau_n}g\left(S^c_{u+\tau_n}\right)du\right]\\
	&\pomu{(*)}{=}{} \sum_{n=0}^\infty \Ee^x \left[\int_0^{E_{n+1}(S^c_{\tau_n})}g\left(S^c_{r+\tau_n}\right)dr\right]\\
	&\pomu{(*)}{=}{} \sum_{n=0}^\infty \Ee^x \left[\int_0^{E_{n+1}(S^c_{\tau_n})}g\left(S^{n+1}_r+S_{\tau_n}^c\right)dr\right]\\
	&\omu{(*)}{=}{}  \sum_{n=0}^\infty \Ee^x \left[\Ee^ {S_{\tau_n}^c} \left(\int_0^{E_{n+1}(S^c_{0})}g\left(S^{n+1}_u\right)du \right)\right].
\end{align*}
In the last step we use the strong Markov property of the censored process. Observe that $E_{n+1}(S^c_{\tau_n})=\tau_{n+1}$. Therefore, we can use  Lemma~\ref{pr-03} for $S=S^{n+1}$, Lemma~\ref{sp-07}.\ref{sp-07-a} and Theorem~\ref{cvp-09} to get
\begin{gather*}
	\Ee^x\left[\int_0^{\tau_\infty}g\left(S^c_t\right)dt\right]
	=  \sum_{n=0}^\infty \Ee^x\left[\rli{0}{}{f} g\left(S^c_{\tau_n}\right)\right]
	= \sum_{n=0}^\infty \Kcal^n \rli{0}{}{f} g(x) 
	= \cei{0}{}{f}g(x).
\end{gather*}
If $g\in C[0,T]$, we know from Theorem~\ref{cvp-09} that $\cei{0}{}{f}g(x)$ is finite, and we can use the above calculation for the positive and negative parts $g^\pm$ of $g$. The claim now follows from the linearity of the fractional integral.
\end{proof}

We can now identify the generator of the censored process. 
\begin{theorem}\label{pr-07} 
	Let $(\bar\mu,k)$ be a positive Sonine pair and $\Lscr(\bar\mu; \lambda)=f(\lambda)/\lambda$, $\Lscr(k; \lambda)=1/f(\lambda)$, where $f$ is a Bernstein function satisfying the Assumptions~\ref{A1} and~\ref{A2}.  Moreover, assume that $\limsup_{x\to 0}\bar{\mu}(x)\int_0^x k(s)\,ds<1$. The process $S^c=(S^c_t)_{t\geq 0}$ is a Feller process. 
	
	For any $T>0$ the semigroup induced by $S^c$ on the Banach space 
	$C_\infty(0, T]=\overline{C_c(0, T]}^{\|\cdot\|_{C[0,T]}} = \left\{u\in C(0, T]\,:\, u(0+)=0\right\}$ has the generator $\left(-\ced{0}{}{f}, \cei{0}{}{f}\left(C_\infty(0,T]\right) \subset C_{\bar\mu}(0,T]\right)$.
\end{theorem}
\begin{proof} 
Since $S^c$ is a Markov process, $P_t^c\phi(x):=\Ee^x[\phi(S_t^c)]$, $x\in (0,T]$ is a positivity preserving contraction semigroup on the Borel--measurable functions $\Bcal(0,T]$. We are going to show that $P_t^c$, $t>0$, is a Feller operator i.e.\ $P_t^c: C_\infty(0,T]\to C_\infty(0,T]$ and that $t\mapsto P_t^c \phi$ is strongly continuous.

\medskip\noindent
\emph{Step~1.} $(P_t^c)_{t\geq 0}$ is strongly continuous  on $C_\infty(0,T]$. Assume that $\phi\in C_\infty(0,T]$.
Then we have 
\begin{align*}
	|P_t^c\phi(x)-\phi(x)|
	&=\left|\Ee^x[\phi(S^c_t)]-\phi(x)\right|\\
	&\leq \left|\Ee^x\left[(\phi(S^c_t)-\phi(x))1_{\{t<\tau_1\}}\right]\right|+\left|\Ee^x \left[(\phi(S^c_t)-\phi(x))1_{\{t\geq \tau_1\}}\right]\right|\\
	&=\left|\Ee \left[(\phi(x-S^1_t)-\phi(x))1_{\{t<\tau_1\}}\right]\right|+\left|\Ee^x \left[(\phi(S^c_t)-\phi(x))1_{\{t\geq \tau_1\}}\right]\right|.
\end{align*}
Since $S^1$ is a subordinator, hence a Feller process, the first term vanishes uniformly in $x$ as $t\to 0$. The second term is further bounded as follows: 
\begin{align*}
	\left|\Ee^x \left[(\phi(S^c_t)-\phi(x))1_{\{t\geq \tau_1\}}\right]\right|
	&\leq 2\|\phi\|_{C[0,x]}\Pp^x\left[t\geq \tau_1\right]\\
	&= 2\|\phi\|_{C[0,x]}\Pp\left[t\geq E_1(x)\right]\\
	&= 2\|\phi\|_{C[0,x]}\Pp^0\left[S^1_t\geq x\right].
\end{align*}
Fix $\epsilon>0$. Since $\phi(0+)=0$, there is some $\delta$ such that for all $0\leq x\leq \delta$ we have $\|\phi\|_{C[0,x]}\leq \epsilon$. If $x>\delta$, we get $\lim_{t\to 0+}\Pp\left(S^1_t\geq x\right)=0$. This shows that
\begin{gather*}
	\|\phi\|_{C[0,x]}\Pp\left[x\leq S^1_t\right] 
	\leq
	\begin{cases} 
		\epsilon, &\text{if\ \ } 0\leq x\leq \delta, \\
		\epsilon \|\phi\|_{C[0,T]}, &\text{if\ \ } \delta<x\leq T. 
 	\end{cases}
\end{gather*}
Since $\epsilon>0$ is arbitrary, we conclude that $(P_t^c)_{t\geq 0}$ is strongly continuous.

\begin{figure}[t]
	\centering    
	\includegraphics[width=0.7\textwidth]{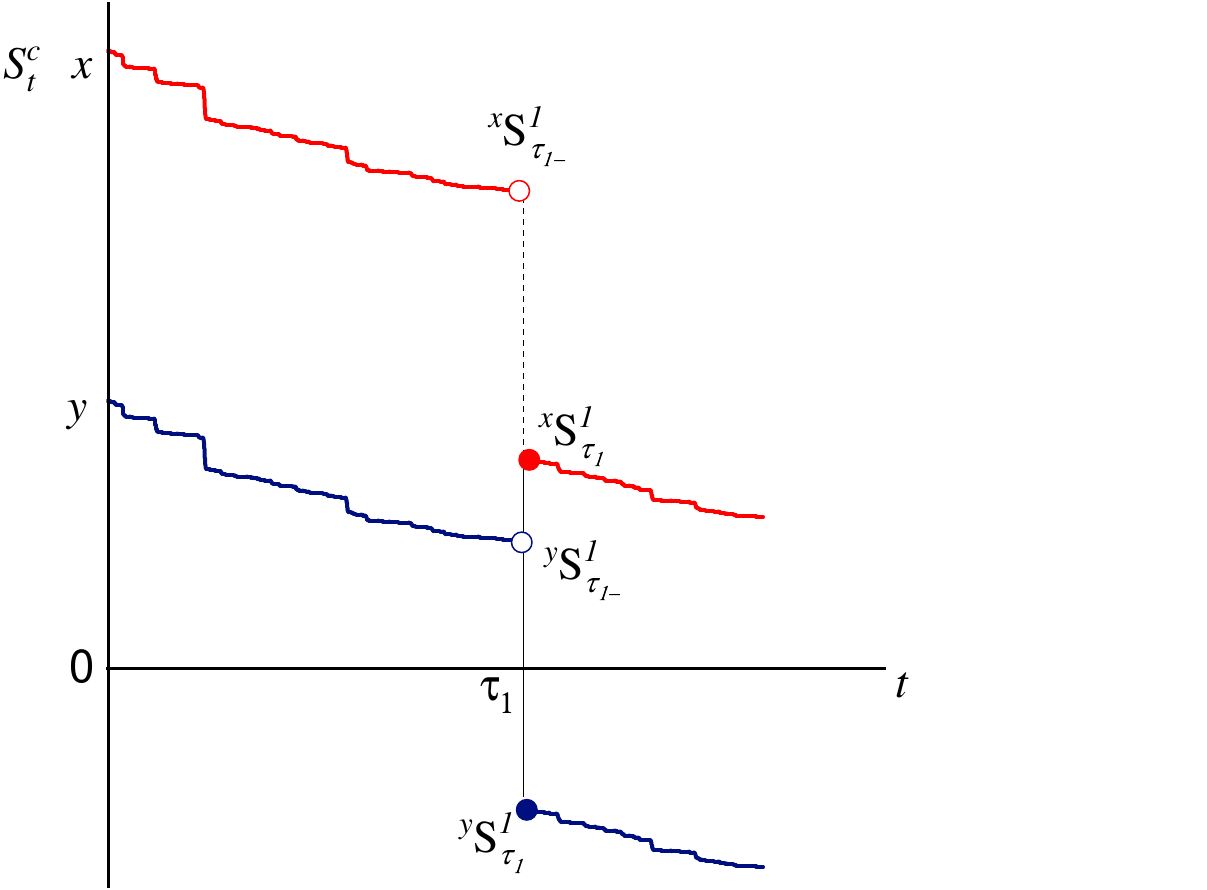}
	\caption{The processes ${}^xS^c_t - {}^yS^c_t$ ``march'' in parallel as long as they are not or always simultaneously censored. This changes at the first censoring event, where only one of the processes is censored. This situation is shown for $\tau_1$, where only ${}^yS^c$ is censored, hence the distance ${}^xS^c_{\tau_1} - {}^yS^c_{\tau_1}<x-y$.}\label{fig}
\end{figure}

\medskip\noindent
\emph{Step~2.} $P_t^c$ is Feller continuous, i.e.\ $P_t^c : C_\infty(0,T]\to C_\infty(0,T]$. We begin by showing that $(P_t^c \phi)(0+)=0$. Note that $\phi\in C_\infty(0,T]$ satisfies $\phi(0+)=0$. The calculations in Step~1  show 
\begin{gather*}
	\left| P_t^c \phi(x)\right|
	\leq |\phi(x)|+ \left|\Ee^0\left[\left(\phi(x-S^1_{t}\right)-\phi(x)) \I_{\{t\leq \tau_1\}}\right]\right| +2\|\phi\|_{C[0, x]}\Pp^0\left(S_t^1\geq x\right)
	\leq 5\|\phi\|_{C[0, x]}. 
\end{gather*}
Since $\lim_{x\to 0} \|\phi\|_{C[0, x]} = \phi(0+)=0$, the claim follows.

Now check that $x\mapsto P_t^c\phi(x)$ is continuous. Pick $\phi\in C_\infty(0, T]$ and assume, without loss of generality, that $0\leq x\leq y$. Since $\phi$ is uniformly continuous, for every $\epsilon>0$ there is some $h>0$ such that $|\phi(x)-\phi(y)|\leq \epsilon$ for all $|x-y|\leq h$. Writing ${}^{z}S^c_t$ for the process $S^c_t$ with starting point $S^c_0=z$, we have
\begin{align*}
	&\left|P_t^c\phi(x)-P_t^c\phi(y)\right|
	=  \left|\Ee^x\phi\left(S^c_t\right)-\Ee^y\phi\left(S_t^c\right)\right|\\
	&\quad\leq \left|\Ee\left[\I_{\left\{\left|{}^xS^c_t-{}^{y}S^c_t\right|\leq h\right\}}\left[\phi({}^xS^c_t)-\phi({}^{y}S^c_t)\right]\right]\right|
	+ \left|\Ee\left[1_{\left\{\left|{}^xS^c_t-{}^{y}S^c_t\right|>h\right\}}\left[\phi({}^xS^c_t)-\phi({}^{y}S^c_t)\right]\right]\right|\\
	&\quad\leq \epsilon +2 \|\phi\|_{C[0,T]} \Pp\left(|{}^xS^c_t-{}^{y}S^c_t|>h\right).
\end{align*}
In order to deal with the second term, we define
\begin{align*}
	\nu(x,y)
	:= \inf\left\{n\in \nat\::\:\tau_n(x)\neq\tau_n(y)\right\}.
\end{align*}
We write $\tau_k:=\tau_k(y)$ for the $k$th censoring time of the process ${}^yS^c$, and calculate
\begin{align}
	\Pp\left(|{}^xS^c_t-{}^{y}S^c_t|>h\right)
	&\notag= \sum_{ k\in \nat}\Pp\left(|{}^xS^c_t-{}^{y}S^c_t|>h,\: \tau_{k}\leq t<\tau_{k+1}\right)\\
	&\notag\leq \sum_{k\in \nat} \Pp\left(|{}^xS^c_t-{}^{y}S^c_t|>h,\: \tau_{k}\leq t<\tau_{k+1},\: \nu(x,y)\leq k\right)\\
	&\notag\mbox{}\qquad +\Pp\left(|{}^xS^c_t-{}^{y}S^c_t|>h,\: \tau_{k}\leq t<\tau_{k+1},\: \nu(x,y)> k\right)\\
	&\notag= \sum_{k\in \nat} \Pp\left(|{}^xS^c_t-{}^{y}S^c_t|>h,\: \tau_{k}\leq t<\tau_{k+1},\: \nu(x,y)\leq k\right)\\
	&\label{pr-e05}\leq \sum_{k\in \nat}\sum_{i=1}^k\Pp\left(|{}^xS^c_t-{}^{y}S^c_t|>h,\: \tau_{k}\leq t<\tau_{k+1},\: \nu(x,y)=i\right). 
\end{align} 
In the second equality we note that on the set $\left\{\tau_{k}\leq t<\tau_{k+1},\: \nu(x,y)> k\right\}$, the processes ${}^xS^c_t$ and ${}^yS^c_t$ march in parallel until $t<\tau_{k+1}(y)$; since $x-y<h$, the estimate $|{}^xS^c_t-{}^{y}S^c_t|>h$ cannot hold, see Fig.~\ref{fig}. Further, using the independence of $S^i$ and ${}^yS^c_{\tau_{i-1}}$ and Lemma~\ref{sp-07}.\ref{sp-07-a}, 
\begin{align*}
	&\Pp\left(|{}^xS^c_t-{}^{y}S^c_t|>h,\tau_{k}\leq t<\tau_{k+1},\nu(x,y)=i\right)\\
	&\qquad\leq \Pp\left(\nu(x,y)=i\right)\\
	&\qquad\leq \Pp\left(S^i_{\sigma_{i}} \in \left({}^yS^c_{\tau_{i-1}},\,{}^yS^c_{\tau_{i-1}}+x-y\right),\: S^i_{\sigma_{i}-}\in \left(0,\,{}^yS^c_{\tau_{i-1}}\right)\right)\\
	&\qquad = \int_0^y \Pp\left(S^i_{\sigma_{i}}\in (z,z+x-y),S^i_{\sigma_{i}-}\in (0,z)\right) \,\Pp\left({}^yS^c_{\tau_{i-1}} \in dz\right)\\
	&\qquad = \int_0^y \Pp\left(S^i_{\sigma_{i}}\in (z,z+x-y),S^i_{\sigma_{i}-}\in (0,z)\right) \,k_{i-1}(y,z)\,dz\\
	&\qquad = \int_0^y \int_0^z \left[\bar\mu(z-a) - \bar\mu(z-a+x-y)\right] k(a)\,da \,k_{i-1}(y,z)\,dz
	\xrightarrow[x\to y]{}0.
\end{align*}
Because of
\begin{gather*}
	\sum_{i=1}^k \Pp\left(|{}^xS^c_t-{}^{y}S^c_t|>h,\:\tau_{k}\leq t<\tau_{k+1},\:\nu(x,y)=i\right)
	\leq \Pp\left(\tau_{k}\leq t<\tau_{k+1}\right)
\end{gather*} 
we can use dominated convergence in \eqref{pr-e05}, and see that $\lim_{x\to y}\Pp\left(|{}^xS^c_t-{}^{y}S^c_t|>h\right)=0$. Changing the roles of $x$ and $y$ in the proof implies that $\lim_{x-y\to 0}\left|P_t^c\phi(x)-P_t^c\phi(y)\right| = 0$, showing that $(P^c_t)_{t\geq 0}$ is a Feller semigroup.

\medskip\noindent
\emph{Step~3.} Identification of the generator. We know from Theorem~\ref{pr-05} that $\left(\cei{0}{}{f},C_\infty(0,T]\right)$ is the potential operator for the Feller semigroup $(P^c_t)_{t\geq 0}$. From Theorem~\ref{cvp-09} we see that $\ced{0}{}{f}: \cei{0}{}{f}\left(C_\infty(0,T]\right) \subset C_{\bar\mu}(0,T]  \to C_\infty(0,T]$ is the inverse operator, hence $-\ced{0}{}{f}$ is the infinitesimal generator of the semigroup.
\end{proof}

\begin{remark}\label{pr-09} 
	Our results also allow for an analytic construction of the censored semigroup $(P^c_t)_{t\geq 0}$ as a Feller semigroup on $C_\infty(0,T]$. The starting point is the generator $-\ced{0}{}{f}$ with domain $\cei{0}{}{f} \left(C_\infty(0,T]\right) \subset C_{\bar\mu}(0,T]\cap C_\infty(0,T]$, cf.\ Theorem~\ref{bd-19}.\ref{bd-19-a}. Since $C_c(0,T]\subset C_{\bar\mu}(0,T]$ it is clear that $-\ced{0}{}{f}$ is a densely defined linear operator in $C_\infty(0,T]$. From the integral representation in Remark~\ref{cvp-05} we see that $-\ced{0}{}{f}\phi(x_0)\leq 0$ at every maximum point $x_0$ of $\phi$. Thus, $-\ced{0}{}{f}$ satisfies the positive maximum principle, which implies dissipativity. Finally, Theorem~\ref{re-05} shows that the range of the operators $\left(\lambda - \ced{0}{}{f}\right)$, $\lambda>0$, is $C_\infty(0,T]$. Therefore, the conditions of the Hille--Yosida--Ray theorem, cf.\ \cite[Thm.~I.2.6, Thm.~IV.2.2]{2009_Ethier} are satisfied and we see that $-\ced{0}{}{f}$ generates a Feller semigroup $(P_t)_{t\geq 0}$. Using the argument of Step~3 in the proof of Theorem~\ref{pr-07} we can identify $P_t$ with the semigroup $P_t^c$ of the censored process $S^c$.
\end{remark}

From standard semigroup theory, see e.g.\ \cite{2009_Ethier}, we know that $\phi(t, x) = \Ee^x[g(S_t^c)]$, $g\in C[0,T]$, is the (unique) solution to the following Cauchy problem:
\begin{equation}\label{pr-e06}
	\left\{\begin{aligned}
	\partial_t \phi(t,x) &=-\ced{0}{}{f} \phi(t,x),\\ 
	\phi(0,x)&=g(x).
	\end{aligned}\right.
\end{equation}

When solving exit problems by computing the Laplace transform of the lifetime of a killed Markov process, we can obtain the analytical solution to the resolvent equation.
\begin{theorem}\label{pr-11}
	Let $(\bar\mu,k)$ be a positive Sonine pair and $\Lscr(\bar\mu; \lambda)=f(\lambda)/\lambda$, $\Lscr(k; \lambda)=1/f(\lambda)$, where $f$ is a Bernstein function satisfying the Assumptions~\ref{A1} and~\ref{A2}.  Moreover, assume that $\limsup_{x\to 0}\bar{\mu}(x)\int_0^x k(s)\,ds<1$.
	Let $T>0$. For any $\lambda>0$, $g\in C[0, T]$ and $x\in (0,T]$ one has
	\begin{equation}\label{pr-e08}
		\Ee^x\left[\int_0^{\tau_\infty} e^{-\lambda t}g(S_t^c)\,dt\right]
		= \sum_{n=0}^\infty (-\lambda)^n \left(\cei{0}{}{f}\right)^{n+1} g(x).
	\end{equation}
	In particular, we have
	\begin{align}\label{pr-e10}
		\Ee^x\left[e^{-\lambda \tau_\infty}\right]
		&=\sum_{n=0}^\infty (-\lambda)^n \left(\cei{0}{}{f}\right)^n \I(x)
	\intertext{and}\label{pr-e12}
		\left(\cei{0}{}{f}\right)^{n+1} g(x)
		&= \Ee^x\left[\int_0^{\tau_\infty} \frac{t^n}{n!} g(S_t^c)\,dt\right].
	\end{align}
\end{theorem}
\begin{proof}
	From Theorem~\ref{pr-07} and Remark~\ref{pr-07} we know that $(P_t^c)_{t\geq 0}$ is the Feller semigroup on $C_\infty(0,T]$ with the generator $-\ced{g}{}{f}$. The Hille--Yosida--Ray theorem shows that the following resolvent equation has for every $g\in C_\infty(0,T]$ a unique solution
\begin{equation*} 
\left\{
\begin{aligned}
	-\ced{0}{}{f}\phi(x) &= \lambda \phi(x)-g(x), &&x\in (0, T],\\
	\phi(x) &=0, &&x=0; 
\end{aligned}
\right.
\end{equation*}
This solution $\phi = (\lambda + \ced{0}{}{f})^{-1}g$ is given by \eqref{pr-e08}: use
\begin{gather*}
	\left(\lambda + \ced{0}{}{f}\right)^{-1}g(x)
	= \int_0^\infty e^{-\lambda t} P_t^c g(x)\,dt
	= \int_0^\infty \Ee^x\left[ e^{-\lambda t} g(S_t^c)\right] dt
	= \Ee^x\left[\int_0^{\tau_\infty}  e^{-\lambda t} g(S_t^c)\, dt\right]
\end{gather*}
in conjunction with Theorem~\ref{re-05} and Remark~\ref{sp-02}.

In order to extend the equality \eqref{pr-e08} to any $g\in C[0, T]$, we take $g_k\in C_\infty(0, T	]$ such that $g_k$ converges to $g$ locally uniformly in $(0, T]$. The integrable majorant $\sup_{k\in\nat}\|g_k\|_{C[0, T]} e^{-\lambda t}\in L^1\left(\Pp^x\otimes dt\right)$ allows us to use dominated convergence to see 
\begin{gather*}
	\lim_{k\to\infty}\Ee^x\left[\int_0^{\tau_\infty}e^{-\lambda t}g_k(S_t^c)\,dt\right]
	= \Ee^x\left[\int_0^{\tau_\infty}e^{-\lambda t}g(S_t^c)\,dt\right].
\end{gather*}
This gives an approximation of the left-hand side \eqref{pr-e08}. 

For the right-hand side, we assume for a moment that $g\geq 0$ and $0\leq g_k\leq g$ increases to $g$. Since $\cei{0}{}{f}$ is  positivity preserving and linear,  we have $\cei{0}{}{f} g_k \uparrow \cei{0}{}{f} g$ as $k\to\infty$. Thus, 
\begin{gather*}
	\lim_{k\to\infty}\sum_{n=0}^\infty |\lambda|^n \left(\cei{0}{}{f}\right)^{n+1} g_k 
	= \sum_{n=0}^\infty |\lambda|^n \left(\cei{0}{}{f}\right)^{n+1} g.
\end{gather*}
The general case follows by considering positive and negative parts: let $g=g^+-g^-$ and take increasing sequences $g_n\to g^+$, $h_n\to g^-$ as $n\to\infty$. Therefore, \eqref{pr-e08} holds for all functions $g\in C[0, T]$.

For $g=\lambda\I$ the left-hand side of \eqref{pr-e08} becomes
\begin{gather*}
	\Ee^x\left[\int_0^{\tau_\infty}e^{-\lambda t}\lambda\,dt\right]
	= \Ee^x \left[1 - e^{-\lambda \tau_\infty}\right],
\end{gather*}
while the right-hand side of \eqref{pr-e08} is $\lambda\sum_{n=0}^\infty (-\lambda)^{n}\left(\cei{0}{}{f}\right)^{n+1}\I = 1-\sum_{n=0}^\infty (-\lambda)^n \left(\cei{0}{}{f}\right)^n \I$.
This proves \eqref{pr-e10}.

Finally, \eqref{pr-e12} follows if we use the exponential series on the left-hand side of \eqref{pr-e08} and compare coefficients of the resulting formal power series.
\end{proof}

\appendix

\section{Positive Sonine pairs and Bernstein functions}\label{app}

In this appendix we use an extended version of Definition~\ref{bd-09}. Let $(g,\nu)$ be a pair consisting of a measurable function $g:(0,\infty)\to [0,\infty)$ and Borel measure on $\left([0,\infty),\mathcal{B}[0,\infty)\right)$, which is finite on compact subsets of $(0,\infty)$. We call $(g,\nu)$ a \textbf{positive Sonine pair}, if the convolution equation
\begin{gather}\label{app-e02}
	g*\nu(x) := \int_{(0,t)} g(x-t)\,\nu(dt) = 1\quad\text{for all\ \ } t\in (0,1)
\end{gather}
holds. Since $g$ is positive and measurable, the convolution $g*\nu$ is always well-defined in $[0,\infty]$.

As the convolution of two measures is, in general, a measure, it is clear that one of the factors in a positive Sonine pair has to be a function. If we approximate $(g,\nu)$ by an increasing sequence of functions $g_n := (n\wedge g)\I_{[n^{-1},n]}\in L^1(dx)$ resp.\ finite measures $\nu_n := \nu\left(\bullet\cap [n^{-1},n]\right)$, we can use monotone convergence to calculate the Laplace transform of $g*\nu = \sup_n \Lscr(g_n*\nu_n)$, and we get
\begin{gather}\label{app-e04}
	\Lscr(g*\nu;\lambda) = \Lscr(g;\lambda)\Lscr(\nu,\lambda) = \Lscr(\I,\lambda)= \frac 1\lambda,\quad \lambda >0.
\end{gather}
From this we see that necessarily $g\in L^1_{\mathrm{loc}}[0,\infty)$ and $\nu(0,1)<\infty$. This proves

\begin{lemma}\label{app-03}
	Let $(g,\nu)$ be a positive Sonine pair. Then $g\in L^1_{\mathrm{loc}}[0,\infty)$, i.e.\ $\int_0^1 g(x)\,dx < \infty$ and $\nu(0,1)<\infty$.
	If $\nu(dx) = h(x)\,dx$, then $h\in L^1_{\mathrm{loc}}[0,\infty)$.
\end{lemma}

We are interested in the relation of positive Sonine pairs and Bernstein functions. Let $f\in\BF$ be a \textbf{Bernstein function} with triplet $(a,b,\mu)$, i.e.\
\begin{equation}\label{app-e06}
	f(\lambda) = a+b \lambda+\int_0^\infty \left(1-e^{-\lambda t}\right) \mu(dt),\quad \lambda>0,
\end{equation}
where $a,b\geq 0$ and $\mu$ is a measure such that $\int_{(0,\infty)} \left(t\wedge 1\right) \mu(dt)<\infty$. If $\mu(dt) = m(t)\,dt$ with a completely monotone function $m\in\CM$, we call $f$ a \textbf{complete Bernstein function} and write $f\in\CBF$. If $f\in\BF$ is such that the conjugate function $f^\star(\lambda):=\lambda/f(\lambda)$, then $f$ is said to be a \textbf{special Bernstein function}, and we write $f\in\SBF$. It is well known, cf.\ \cite{2012_Schilling}, that $\CBF\subsetneqq\SBF$ as well as
\begin{gather*}
	f\in\CBF \iff f^\star\in\CBF
	\quad\text{and}\quad
	f\in\SBF \iff f^\star\in\SBF.
\end{gather*}

\begin{theorem}\label{app-05} 
	Let  $f\in \SBF$ be a special Bernstein function with triplet $(a, b, \mu)$, $\bar\mu(x):=\mu[x, \infty)$, and $f^\star\in\SBF$ its conjugate function with triplet $(a^\star, b^\star, \mu^\star)$, $k(x):=\mu^\star [x, \infty)$. Then 
\begin{gather}\label{app-e08}\begin{aligned}
	\frac{f(\lambda)}{\lambda}	&= \Lscr\left(a+\bar\mu+b\delta_0;\lambda\right),\\
	\frac{1}{f(\lambda)} 		&= \Lscr\left(a^\star+k+b^\star\delta_0;\lambda\right).
\end{aligned}\end{gather}
Moreover, using the convention that $1/\infty=0$, one has
\begin{gather}\label{app-e10}
	f^\star(0)
	= a^\star 
	= \lim_{\lambda \to 0} \frac{\lambda}{f(\lambda)}
	= \begin{cases}
		0, &\quad a>0, \\
		\frac 1{b+\int_0^\infty t \mu(dt)}, &\quad a=0,
	\end{cases}
\\\label{app-e12}
	b^\star
	= \lim_{\lambda\to \infty}\frac{ f^\star (\lambda)}{\lambda}
	= \lim_{\lambda\to \infty}\frac{1}{f(\lambda)}
	= \begin{cases}
		0, &\quad b>0,\\
		\frac 1{a+\int_0^\infty \mu(dt)}, &\quad b=0,
	\end{cases}
\end{gather}
and $(a+\bar\mu)+b\delta_0$ and $(a^\star+k)+b^\star\delta_0$ is positive Sonine pair \textup{(}note that $aa^\star = bb^\star = 0$\textup{)}.  If $b=b^\star=0$, then $a+\bar\mu$ and $a^\star+k$ is a positive Sonine pair of decreasing functions.
\end{theorem}
\begin{proof}
We have  $f(\lambda)=a+b\lambda +\int_{(0,\infty)} \left(1-e^{-\lambda t}\right) \mu(dt)$, and by the Fubini--Tonelli theorem 
\begin{gather}\label{app-e14}\begin{aligned}
	\frac{f(\lambda)}{\lambda}
	&= \frac{a}{\lambda}+b+\int_{(0,\infty)} \frac{1-e^{-\lambda t}}{\lambda}\,\mu(dt)\\
	&= \frac{a}{\lambda}+b+\int_{(0,\infty)} \int_0^t e^{-\lambda t}\,dt\,\mu(dt)\\
	&= \frac{a}{\lambda}+b+\int_0^\infty \int_{(x,\infty)} \mu(dt)\, e^{-\lambda t}\,dt\\
	&= \frac{a}{\lambda}+b+\int_0^\infty\bar\mu(t) e^{-\lambda t}\,dt\\
	&= \Lscr(a+b\delta_0+\bar\mu; \lambda).
\end{aligned}\end{gather}
Since $f\in \SBF$, we have $f^\star\in\SBF$, and writing $(a^\star, b^\star,\mu^\star)$ for its triplet, we get 
\begin{gather*}
	\frac1{f(\lambda)}
	= \frac{f^\star(\lambda)}{\lambda}
	= \Lscr\left(a^\star+b^\star\delta_0+k; \lambda\right).
\end{gather*}
Since  $\frac{1}{\lambda}=\frac{1}{f(\lambda)} \frac{f(\lambda)}{\lambda}$, it is obvious that $a+b\delta_0+\bar\mu$ and $a^\star+b^\star\delta_0+k$ is a Sonine pair (note that $bb^\star=0$, i.e.\ at least one Sonine factor is a function) and the relations \eqref{app-e10} and \eqref{app-e12} follow from \cite[Ch.~11]{2012_Schilling} or Theorem 2.3.11 in \cite{2023_Li}.  

Note that (see Table \ref{tab:table1}) $b=b^\star=0$ if, and only if, we are in the case (1), (2), (4) as shown in Table \ref{tab:table1}. Since $\bar\mu(x) = \mu(x,\infty)$ and $k(x) = \mu^\star(x,\infty)$, it is clear that $a+\bar\mu$ and $a^\star+k$ are decreasing functions.
\end{proof}

\renewcommand{\arraystretch}{1.5}
\begin{table}[h!]
\begin{center}
\caption{Overview of all cases covered by \eqref{app-e08} and \eqref{app-e10}.}\label{tab:table1}
\begin{tabular}{c|c|c|c|c|c|c}
Nr. & $a$ & $b$ & $m_0=\int_{(0,\infty)} \mu(dt)$ & $m_1=\int_{(0,\infty)} t \,\mu(dt)$ & $a^\star$ & $b^\star$\\ 
\hline
(1) & $0$ & $0$ & --- (no condition)& $\infty$ & $0$ &$0$\\
(2) & $0$ & $0$ & $\infty$ & $<\infty$ &$1/m_1$ & $0$\\
(3) & $0$ & $0$ & $<\infty$ & $<\infty$ & $1/m_1$ & $1/m_0$\\
(4) & $>0$ & $0$ & $\infty$ & --- & $0$ & $0$\\
(5) & $>0$ & $0$ & $<\infty$ & --- & $0$ & $1/(a+m_0)$\\
(6) & $0$ & $>0$ & --- & $\infty$ & $0$ & $0$\\
(7) & $0$ & $>0$ & --- & $<\infty$ & $1/(b+m_1)$ & $0$\\
(8) & $>0$ & $>0$ & --- & --- & $0$ & $0$
\end{tabular}
\end{center}
\end{table} 

\begin{corollary}\label{app-07}
	If, in the setting of Theorem~\ref{app-05}, $f\in\CBF$, then $f^\star\in\CBF$ and $a+\bar\mu$ and $a^\star+k$ are completely monotone functions. In particular, if $b=b^\star=0$, then $(a+\bar\mu,a^\star+k)$ is a positive Sonine pair of completely monotone functions. 
\end{corollary}
\begin{proof}
	This follows from the proof of Theorem~\ref{app-05}: observe that $f\in\CBF$ entails that the functions $f(\lambda)/\lambda$ and $1/f(\lambda)$ appearing in \eqref{app-e08} are Stieltjes functions, implying that $\bar\mu, k\in\CM$, cf.\ \cite[Ch.~7]{2012_Schilling}.
\end{proof}

\begin{remark}\label{app-09}
	The standing assumption \ref{A1}, see page \pageref{A2}, of the main text ensures that $a=a^\star=0$ and $b=b^\star=0$.
	
	In this case, Theorem~\ref{app-05} (resp.\ Corollary~\ref{app-07}) is actually a one-to-one correspondence between $\SBF$ (resp.\ $\CBF$) and positive Sonine pairs comprising of decreasing (resp.\ completely monotone) functions of the form $\mu(x,\infty)$, $\mu^\star(x,\infty)$ where $(a,b,\mu)$ and $(a^\star,b^\star,\mu^\star)$ are triplets of conjugate Bernstein functions. 
	
	This correspondence remains valid in the case when $aa^\star = 0 = bb^\star$, but then one factor of the Sonine pair will have a $\delta_0$-component. In general, there are further positive Sonine pairs related to Bernstein functions. Consider, for example, a completely monotone function $\Lscr(u;\lambda)$ where $u:(0,\infty)\to\real$ is a continuous function, which is not decreasing, such that $\Lscr(u;\lambda) = 1/f(\lambda)$ is the potential of a subordinator with Bernstein function $f\in\BF$ with $b>0$, see e.g.\ Bertoin \cite[Prop.~1.7]{1999_Bertoin}. Since $f(\lambda)/\lambda$ is completely monotone, we get
\begin{gather*}
	\frac 1\lambda = \frac 1{f(\lambda)} \frac{f(\lambda)}{\lambda} = \Lscr(u;\lambda)\Lscr(\nu;\lambda) = \Lscr(u*\nu;\lambda),
\end{gather*}
i.e.\ $(u,\nu)$ is a positive Sonine pair induced by the Bernstein function $f$. Since $\lambda \Lscr(\nu,\lambda) = f(\lambda)$, the calculation \eqref{app-e14} reveals that
\begin{gather*}
	\lambda\Lscr\left(\nu;\lambda\right)
	= f(\lambda)
	= \lambda\Lscr\left(a + b\delta_0 + \bar\mu;\lambda\right),
\end{gather*}
i.e.\ $\nu(dx) = b\delta_0(dx) + (a+\bar\mu(x))\,dx$. Notice that $f$ cannot be a special Bernstein function.
\end{remark}

\section*{Declaration of interests}
{The authors report no conflict of interest.}


\frenchspacing

\begin{thebibliography}{10}
	
\bibitem{1996_Bertoin}
Bertoin, J: \emph{L\'evy processes}.
Cambridge University Press, Cambridge Tracts in Mathematics vol.\ 121, Cambridge 1996.


\bibitem{1999_Bertoin}
Bertoin, J: Subordinators: examples and applications.
In: \emph{Lectures on Probability Theory and Statistics -- Ecole d'\'Et\'e de Probabilit\'es de Saint-Flour XXVII--1997}.
Springer, Lecture Notes in Mathematics vol.\ 1717, Berlin 1999.

\bibitem{2017_Chen}
Chen, Z-Q.: Time fractional equations and probabilistic representation.
\emph{Chaos, Solitons \& Fractals} \textbf{102} (2017) 168--174.
	
\bibitem{2010_Diethelm}
Diethelm, K.: \emph{The Analysis of Fractional Differential Equations: An Application-oriented Exposition Using Differential Operators of Caputo Type}.
Springer, Lecture Notes in Mathematics vol.\ 2004, Berlin 2010.
	
\bibitem{2021_Du}
Du, Q.; Toniazzi, L.; Xu, Z: Censored stable subordinators and fractional derivatives.
\emph{Fractional Calculus and Applied Analysis} \textbf{24} (2012) 1035--1068.
	
\bibitem{2009_Ethier}
Ethier, S.N.; Kurtz T.G.: \emph{Markov processes: characterization and convergence}.
John Wiley \& Sons, Wiley Series in Probability and Mathematical Statistics, New York 1986.
	
\bibitem{2020_Hanyga}
Hanyga, A.: A comment on a controversial issue: A generalized fractional derivative cannot have a regular kernel.
\emph{Fractional Calculus and Applied Analysis} \textbf{23} (2020) 211--223.
	
\bibitem{2018_HernandezHernandez}
Hern\'andez-Hern\'andez, M.E.; Kolokoltsov V.N.: Probabilistic solutions to nonlinear fractional differential equations of generalized {C}aputo and {R}iemann–{L}iouville type.
\emph{Stochastics} \textbf{90} (2018) 224--255.
	
\bibitem{2016_HernandezHernandez}
Hern\'andez-Hern\'andez, M.E.; Kolokoltsov, V.N.: On the solution of two-sided fractional ordinary differential equations of {C}aputo type.
\emph{Fractional Calculus and Applied Analysis} \textbf{19} (2016) 1393--1413.
	
\bibitem{1966_Nobuyuki}
Ikeda, N.; Nagasawa, M.; Watanabe, S.: A construction of {M}arkov processes by piecing out.
\emph{Proceedings of the Japan Academy} \textbf{42} (1966) 370--375.
	
\bibitem{2008_Klages}
Klages, R.; Radons, G.; Sokolov, I.M: \emph{Anomalous Transport}.
Wiley-VCH, Weinheim 2008.
	
\bibitem{2011_Kochubei}
Kochubei, A.N.: General fractional calculus, evolution equations, and renewal processes.
\emph{Integral Equations and Operator Theory} \textbf{71} (2011) 583--600.
	
\bibitem{2023_Li}
Li, C.:
\emph{Fractional Time Derivatives and Stochastic Processes}.
Phd thesis, Technische Universit{\"{a}}t Dresden, 2023.
\url{nbn-resolving.org/urn:nbn:de:bsz:14-qucosa2-901780}.
	
\bibitem{2022_Luchko}
Luchko, Y.: The 1st level general fractional derivatives and some of their properties.
\emph{Journal of Mathematical Sciences} \textbf{266} (2022) 709--722.
	
\bibitem{2007_Mainardi}
Mainardi, F.; Gorenflo, R.: Time-fractional derivatives in relaxation processes: A tutorial survey.
\emph{Fractional Calculus and Applied Analysis} \textbf{10} (2007) 269--308.
	
\bibitem{2019_Meerschaert}
Meerschaert, M.M.; Sikorskii, A.: \emph{Stochastic models for fractional calculus}.
De Gruyter, Studies in Mathematics vol.\ 43 Berlin 2019 (2nd ed).
	
\bibitem{1998_Podlubny}
Podlubny, I.: \emph{Fractional Differential Equations}.
Academic Press, Mathematics in Science and Engineering vol.\ 198, San Diego 1999.
	
\bibitem{2002_Samko}
Samko, S.G.; Cardoso, R.P.: Integral equations of the first kind of {S}onine type.
\emph{International Journal of Mathematics and Mathematical Sciences} \textbf{57} (2003) 3609--3632.
	
\bibitem{2003_Samko}
Samko, S.G.; Cardoso, R.P.: Sonine integral equations of the first kind in $l^ p$.
\emph{Fractional Calculus and Applied Analysis} \textbf{6} (2003) 235--258.
	
\bibitem{1993_Samko}
Samko, S.G.; Kilbas, A.A.; Marichev, O.I.: \emph{Fractional Integrals and Derivatives: Theory and Applications}.
Gordon and Breach, Amsterdam 1993.
	
\bibitem{1999_Satoa}
Sato, K.: \emph{L\'evy Processes and Infinitely Divisible Distributions. Revised Edition}.
Cambridge University Press, Cambridge Studies in Advanced Mathematics vol.\ 68, Cambridge 2013.
	
\bibitem{1998_Schillinga}
Schilling, R.L.: Subordination in the sense of bochner and a related functional calculus.
\emph{Journal of the Australian Mathematical Society} \textbf{64} (1998) 368--396.
	
\bibitem{2012_Schilling}
Schilling, R.L.; Song, R.; Vondra\v{c}ek, Z.: \emph{Bernstein Functions: Theory and Applications}. 
De Gruyter, Studies in Mathematics vol.~37, Berlin 2012 (2nd ed).
	
\bibitem{1988_Sharpe}
Sharpe, M: \emph{General Theory of Markov Processes}.
Academic Press, Pure and Appliede Mathematics vol.\ 133, Boston 1988.
	
\bibitem{2018_Sin}
Sin, C-S.: Well-posedness of general {C}aputo-type fractional differential equations.
\emph{Fractional Calculus and Applied Analysis} \textbf{21} (2018) 819--832.
	
\bibitem{2015_Toaldo}
Toaldo, B.: Convolution-type derivatives, hitting-times of subordinators and time-changed $C_0$-semigroups.
\emph{Potential Analysis} \textbf{42} (2015) 115--140.
	
\end{thebibliography}
\end{document}